\documentclass[12pt]{amsart}

\usepackage[cp1251]{inputenc}
\usepackage[T2A]{fontenc}
\usepackage{amsfonts}
\usepackage{amsbsy}
\usepackage{amssymb}
\usepackage[english]{babel}
\usepackage{hyperref}

\renewcommand{\ge}{\geqslant}
\renewcommand{\le}{\leqslant}

\newcommand{\Ker}{\operatorname{Ker}}
\newcommand{\rk}{\operatorname{rk}}
\newcommand{\oSL}{\operatorname{SL}}
\newcommand{\oSp}{\operatorname{Sp}}
\newcommand{\Spin}{\operatorname{Spin}}
\newcommand{\SO}{\operatorname{SO}}
\newcommand{\diag}{\operatorname{diag}}

\newtheorem{theorem}{Theorem}

\newtheorem{lemma}{Lemma}

\theoremstyle{definition}

\newtheorem*{dfn*}{Definition}

\theoremstyle{remark}

\newtheorem{remark}{Remark}

\makeatletter

\renewcommand{\section}{\@startsection{section}{1}{0pt}%
{3ex plus .5ex minus .2ex}{2ex plus .2ex}%
{\center\normalfont\scshape}}

\renewcommand{\subsection}{\@startsection{subsection}{2}{0pt}%
{2.5ex plus .5ex minus .2ex}{-0ex plus .2ex}%
{\normalfont\bfseries}}

\renewcommand{\subsubsection}{\@startsection{subsubsection}{3}{0pt}%
{2.5ex plus .5ex minus .2ex}{1ex plus .2ex}%
{\center\normalfont\bfseries}}

\makeatother

\begin{document}

\renewcommand{\proofname}{Proof}
\renewcommand{\abstractname}{Abstract}
\renewcommand{\refname}{References}
\renewcommand{\figurename}{Figure}
\renewcommand{\tablename}{Table}

\title[Extended weight semigroups]{Extended weight semigroups\\ of affine spherical homogeneous spaces\\ of non-simple semisimple algebraic groups}

\author{Roman Avdeev}

\thanks{Partially supported by the Russian President's Program `Support of Leading Scientific Schools' (grant no. NSh-1983.2008.1)}

\address{Chair of Higher Algebra, Department of Mechanics and Mathematics,
Moscow State University, 1, Leninskie Gory, Moscow, 119992, Russia}

\email{suselr@yandex.ru}




\begin{abstract}
The extended weight semigroup of a homogeneous space $G/H$ of a
connected semisimple algebraic group $G$ characterizes the spectra
of the representations of $G$ on the spaces of regular sections of
homogeneous linear bundles over $G/H$, including the space of
regular functions on $G/H$. We compute the extended weight
semigroups for all strictly irreducible affine spherical homogeneous
spaces $G/H$, where $G$ is a simply connected non-simple semisimple
complex algebraic group and $H$ a connected closed subgroup of it.
In all the cases we also find the highest weight functions
corresponding to the indecomposable elements of this semigroup.
Among other things, our results complete the computation of the
weight semigroups for all strictly irreducible simply connected
affine spherical homogeneous spaces of semisimple complex algebraic
groups.
\end{abstract}

\maketitle

\section{Introduction}
\label{sec1}

\subsection{}
\label{ssec1.1} Let $G$ be a connected reductive algebraic group
over $\mathbb C$ and $H$ a closed subgroup of it. We consider the
homoge\-ne\-ous space $G/H$ and the algebra $\mathbb C[G/H] =\mathbb
C[G]^H$ of regular functions on it. This algebra has the structure
of a rational $G$-module with respect to the action of $G$ by left
multiplication and decomposes into a direct sum of
finite-dimensional irreducible $G$-modules. A problem of interest is
to find all $\lambda$ for which this decomposition contains the
irreducible $G$-module $V_\lambda$ with highest weight $\lambda$ and
to determine the multiplicity $m_\lambda$ of $V_\lambda$. Those
dominant weights $\lambda$ of $G$ satisfying $m_\lambda\ge1$ form a
semigroup called the \textit{weight semigroup} of the homogeneous
space $G/H$. We denote this semigroup by $\Gamma(G/H)$.

The subgroup $H$ (resp. the homogeneous space $G/H$, the pair
$(G,H)$) is said to be \textit{spherical} if a Borel subgroup
$B\subset G$ has an open orbit in $G/H$. The results of the
paper~\cite{1}, which are discussed in~\S\,\ref{ssec1.3} below,
yield that for spherical homogeneous spaces the $G$-module $\mathbb
C[G/H]$ is \emph{multiplicity free}. (The converse is also true if
$G/H$ is quasi-affine.) By definition, this means that $m_\lambda
\le 1$ for every $\lambda$. For semisimple $G$, the classification
of affine spherical homogeneous spaces (that is, with reductive $H$)
up to local isomorphism was obtained in~\cite{2}--\cite{4}. Namely,
all pairs $(G,H)$ with $G$ a simply connected simple algebraic group
and $H$ a connected reductive spherical subgroup of it are found
in~\cite{2} along with a description of the corresponding weight
semigroup for every such pair. Up to local isomorphism, all affine
spherical homogeneous spaces of non-simple semisimple algebraic
groups are classified in~\cite{3},~\cite{4} (see also~\cite{5} for a
more accurate formulation). Every such space can be obtained by a
certain procedure starting from strictly irreducible spherical
homogeneous spaces. (Their definition will be given
in~\S\,\ref{ssec1.4}.) Up to local isomorphism, these can be of the
following three types:

1) spherical spaces $G/H$, where $G$ is simple (classified
in~\cite{2});

2) spaces $G/H$, where $G = H\times H$, the subgroup $H$ is embedded
in $G$ diagonally, $H$ is simple;

3) spaces $G/H$ corresponding to the pairs $(G,H)$ in
Table~\ref{tab1} below (classified in~\cite{3},~\cite{4}).

The weight semigroups of all spaces of type~1) are computed
in~\cite{2}. For spaces of type~2), the semigroups $\Gamma(G/H)$ are
well known, see~\S\,\ref{ssec1.4}. The weight semigroups for all
spaces of type~3) are computed in the present paper.

More precisely, for each spherical homogeneous space in
Table~\ref{tab1} we compute its extended weight semigroup (to be
precisely defined in~\S\,\ref{ssec1.2}) and the weight functions
corresponding to the indecomposable elements of this semigroup. In
particular, using these results one can easily compute the weight
semigroups for the spaces in Table~\ref{tab1}. Yu.\,V.~Dzyadyk
reported to E.\,B.~Vinberg that he had computed the extended weight
semigroups of all the homogeneous spaces appearing in
Table~\ref{tab1} in~1985. He used another method going back to his
papers~\cite{6},~\cite{7} (see also~\cite{8}) dealing with the case
of symmetric spaces. This method does not require explicit
computation of weight functions. Unfortunately, the results of
Dzyadyk have not been published hitherto.

\subsection{}
\label{ssec1.2} Throughout this paper the base field is the field
$\mathbb C$ of complex numbers, all topological terms relate to the
Zariski topology, all groups are supposed to be algebraic and their
subgroups closed. For any group~$L$ let~$\mathfrak{X}(L)$ denote the
group of its characters in additive notation.

In what follows, we keep the notation $G$ for a connected semisimple
algebraic group. We suppose that a Borel subgroup $B\subset G$ and a
maximal torus $T\subset B$ are fixed. The maximal unipotent subgroup
of $G$ contained in $B$ is denoted by $U$. We identify the group
$\mathfrak X(B)$ with $\mathfrak X(T)$ by restricting the characters
from $B$ to $T$. The set of dominant weights of $B$ is denoted by
$\mathfrak X_+(B)$.

The actions of~$G$ on itself by left multiplication ($(g,x)\mapsto
gx$) and right multiplication ($(g,x)\mapsto xg^{-1}$) induce its
representations on the space~$\mathbb{C}[G]$ of all regular
functions on $G$ given by $(gf)(x)=f(g^{-1}x)$ and $(gf)(x)=f(xg)$,
respectively. For short, we call them the \textit{left action} and
the \textit{right action} respectively. For any subgroup $L\subset
G$, we write~${}^L\mathbb{C}[G]$ (resp. $\mathbb{C}[G]^L$) for the
algebra of functions in~$\mathbb{C}[G]$ that are invariant under the
left (resp. right) action of $L$.

We now introduce the notion of the extended weight semigroup of a
homogeneous space $G/H$.

Let $H\subset G$ be an arbitrary subgroup. For every character
$\chi\in\mathfrak{X}(H)$ we denote by~$V_{\chi}$ the subspace
of~$\mathbb{C}[G]$ consisting of weight functions of weight~$\chi$
with respect to the right action of $H$, that is,
$$
V_{\chi}=\bigl\{f\in \mathbb{C}[G]\colon f(gh)=\chi(h)f(g)\ \
\forall\,g\in G, \ \forall\,h\in H\bigr\}.
$$
Let $H_0\subset G$ be the intersection of the kernels of all
characters in~$\mathfrak X(H)$: $H_0=\bigcap
\limits_{\chi\in\mathfrak{X}(H)}\Ker\chi$. Then $\bigoplus
\limits_{\chi\in\mathfrak{X}(H)}V_\chi=\mathbb{C}[G]^{H_0}$. Since
the left and right actions of~$G$ on~$\mathbb{C}[G]$ commute, the
subspace $V_\chi$ is invariant under the left action of~$G$ for any
$\chi\in\mathfrak{X}(H)$.

We note that the quotient group $H/H_0$ is commutative (that is,
$H/H_0$ is a quasi-torus) and the natural embedding
$\mathfrak{X}(H/H_0)\hookrightarrow\mathfrak{X}(H)$ is an
isomorphism.

Suppose that for each subspace~$V_\chi$ its decomposition into a
direct sum of irreducible $G$-modules is fixed. Then the highest
weight vectors of all these $G$-modules taken over all $\chi$ form a
basis of the algebra (considered as a vector space over $\mathbb C$)
$$
A=A(G/H):={}^U(\mathbb{C}[G]^{H_0})={}^U\mathbb{C}[G]^{H_0}.
$$

Suppose $f\in A\setminus\{0\}$, $\lambda\in\mathfrak{X}_+(T)$, and
$\chi\in\mathfrak{X}(H)$. We say that $f$ is a \textit{weight
function with respect to $B\times H$} (or simply a \textit{weight
function}) \textit{of weight $(\lambda,\chi)$} (in the case
$\mathfrak{X}(H)=0$ we simply write $\lambda$ instead of
$(\lambda,0)$), if $f$ is the highest weight vector of an
irreducible $G$-module $N\subset V_\chi$ with highest
weight~$\lambda$. We denote by~$A_{\lambda,\chi}$ the subspace
of~$A$ consisting of zero and all weight functions of weight
$(\lambda,\chi)$. If $f_1\in A_{\lambda_1,\chi_1}$ and $f_2\in
A_{\lambda_2,\chi_2}$, then $f_1f_2 \in
A_{\lambda_1+\lambda_2,\chi_1+\chi_2}$. Thus the set of pairs
$(\lambda,\chi)$ with $\lambda\in\mathfrak{X}_+(B)$,
$\chi\in\mathfrak{X}(H)$, and $A_{\lambda,\chi}\ne0$ is a semigroup.
We call this semigroup the \textit{extended weight semigroup} (the
term is suggested by Vinberg) of the homogeneous space $G/H$ and
denote it by~$\widehat\Gamma(G/H)$.

\begin{remark}
\label{rm1} We have $\Gamma(G/H)=\{(\lambda,\chi) \in
\widehat\Gamma(G/H)\colon \chi=0\} \subset \widehat\Gamma(G/H)$.
Clearly, $\mathfrak{X}(H)=0$ implies that
$\Gamma(G/H)=\widehat\Gamma(G/H)$.
\end{remark}

\begin{remark}
\label{rm2} The map $\pi\colon\widehat\Gamma(G/H)\to\Gamma(G/H_0)$,
$(\lambda,\chi)\mapsto\lambda$, is surjective, and
$\pi^{-1}(0)=\{(0,0)\}$.
\end{remark}

\begin{remark}
\label{rm3} There is another interpretation of the semigroup
$\widehat\Gamma(G/H)$ in the case when $H$ is connected. Namely,
consider the group $\widehat G=G\times (H/H_0)$ and its subgroup
$\widehat H\simeq H$ embedded in~$\widehat G$ via $h\mapsto
(h,hH_0)$. The algebra $\mathbb{C}[G]^{H_0}$ is a $\widehat
G$-module with respect to the left action of~$G$ and the right
action of~$H/H_0$. Consider the algebra $\mathbb{C}[\widehat
G]^{\widehat H}$ as a $\widehat G$-module with respect to the left
action of~$\widehat G$. The map $\psi\colon\mathbb{C}[\widehat
G]^{\widehat H}\to\mathbb{C}[G]^{H_0}$, $(\psi f)(g)=f(g,eH_0)$, is
an isomorphism of $\widehat G$-modules. (The inverse map $F\mapsto
\widehat F$ is given by $\widehat F(g,hH_0)=F(gh^{-1})$.) Therefore
there is a semigroup isomorphism
$\widehat\Gamma(G/H)\simeq\Gamma(\widehat G/\widehat H)$.
\end{remark}

\subsection{}
\label{ssec1.3} Now suppose that $H\subset G$ is a spherical
subgroup. According to Theorem~1 of~\cite{1}, this implies that the
representation of~$G$ on~$V_{\chi}$ is multiplicity free for each
$\chi\in\mathfrak{X}(H)$. Hence, any pair $(\lambda,\chi)\in
\widehat\Gamma(G/H)$ determines a unique, up to multiplication by a
non-zero constant, weight function $f\in A_{\lambda,\chi}$; that is,
$\dim A_{\lambda,\chi}=1$.

Let us prove the following fact. (Cf. \cite{9}, Proposition~2.)

\begin{theorem}
\label{th1} Suppose~$G$ is simply connected and $H\subset G$ is a
connected spherical subgroup. Then the algebra~$A$ is factorial and
the semigroup $\widehat\Gamma(G/H)$ is free.
\end{theorem}

To prove Theorem~\ref{th1}, we need the following

\begin{theorem}[{\rm \cite{10}, Theorem~3.17}]
\label{th2} Suppose a regular action of an algebraic group $L$ on an
affine variety $X$ is given. If the algebra $\mathbb C[X]$ is
factorial and $L$ is connected and has no non-trivial characters,
then the algebra $\mathbb C[X]^L$ is also factorial. Moreover, for
any element $f\in \mathbb C[X]^L$ all its divisors in $\mathbb C[X]$
are contained in $\mathbb C[X]^L$.
\end{theorem}

\begin{proof}[Proof of Theorem~\ref{th1}]
Since~$G$ is simply connected, the algebra~$\mathbb{C}[G]$ is
factorial (see~\cite{11}, Corollary from Proposition~1). Let
$H=H_u\leftthreetimes R$ be a Levi decomposition of~$H$, where $H_u$
(resp. $R$) is the unipotent radical (resp. a maximal reductive
subgroup) of $H$. Then $H_0=H_u\leftthreetimes [R,R]$ where $[R,R]$
is the derived subgroup of~$R$. Since~$H$ is connected, it follows
that $R$ is also connected, whence the group $[R,R]$ is either
connected and semisimple or trivial. Therefore~$H_0$ is connected
and has no non-trivial characters, so the group $U\times H_0$ also
possesses these properties. Hence one can apply Theorem~\ref{th2} to
the action of~$U\times H_0$ on~$G$, where $U$ and~$H_0$ act on~$G$
by left and right multiplication, respectively. Thus the algebra~$A$
is factorial. Let $\{\mu_i\}$ be the set of all indecomposable
elements of $\widehat\Gamma(G/H)$. As mentioned above, the
sphericity of~$H$ implies that for every element~$\mu_i$ there is a
unique, up to multiplication by a non-zero constant, weight function
$f_i\in A$ of weight~$\mu_i$. Moreover, $f_i$ is an irreducible
element of~$A$. Assume that there is a non-trivial relation
$\sum_ik_i\mu_i=\sum_jl_j\mu_j$ for some integers $k_i,l_j>0$. It
implies the relation $c\prod_i f_i^{k_i}=\prod_j f_j^{l_j}$ for some
$c\in\mathbb{C}^{\times}$, which contradicts the factoriality
of~$A$.
\end{proof}

\subsection{}
\label{ssec1.4} A direct product of spherical homogeneous spaces
$$
(G_1/H_1)\times (G_2/H_2)=(G_1\times G_2)/(H_1\times H_2)
$$
is again a spherical homogeneous space. Spaces of this kind, as well
as spaces locally isomorphic to them are called \textit{reducible},
all others are said to be \textit{irreducible}. A spherical space
$G/H$ is said to be \textit{strictly irreducible} if the spherical
space $G/N(H)^0$ is irreducible, see~\cite{12}, 1.3.6. (Here
$N(H)^0$ is the connected component of the identity of the
normalizer of~$H$ in~$G$.)

\renewcommand{\textfraction}{0}
\renewcommand{\bottomfraction}{1}
\renewcommand{\topfraction}{1}
\begin{table}[h]
\renewcommand{\arraystretch}{1.2}
\renewcommand{\tabcolsep}{1pt}
\begin{center}
\caption{} \vskip3mm \label{tab1} \footnotesize{
\begin{tabular}{|c|c|c|c|c|c|}
\hline
No. & $G\supset H$ & \begin{tabular}{c} \scriptsize{Embedding}\\[-1mm] \scriptsize{diagram} \end{tabular}
&$\rk \widehat\Gamma(G/H)$ &$\widehat\Gamma(G/H)$ &
\begin{tabular}{c} \scriptsize{Note} \end{tabular}
\\
\hline 1 &$\oSL_n{\times} \oSL_{n+1}\supset\oSL_n{\times}\mathbb
C^{\times}$ &\begin{tabular}{c}\begin{picture}(13,23)
\put(2,18){\circle{3}} \put(2,15){\line(0,-1){10}}
\put(2,2){\circle*{3}} \put(4,5){\line(1,2){5}}
\put(11,18){\circle{3}}\put(11,15){\line(0,-1){10}}
\put(11,2){\circle*{3}}
\end{picture}
\end{tabular}
&$2n$
&\begin{tabular}{c} $(\varphi_1,n\chi_0)$,\\
$(\pi_{n-1}{+}\varphi_2,(n{-}1)\chi_0)$,\\
$\dots,(\pi_1{+}\varphi_n,\chi_0)$,\\
$(\pi_{n-1}{+}\varphi_1,-\chi_0),\dots,$\\
$(\pi_1{+}\varphi_{n-1},-(n{-}1)\chi_0)$,\\
$(\varphi_n,-n\chi_0)$
\end{tabular}
&$n\geqslant 2$
\\
\hline 2 &$\Spin_n{\times} \Spin_{n+1}\supset\Spin_n$
&\begin{tabular}{c}
\begin{picture}(22,23)
\put(2,18){\circle{3}} \put(4,15){\line(1,-2){5}}
\put(11,2){\circle*{3}} \put(13,5){\line(1,2){5}}
\put(20,18){\circle{3}}
\end{picture}
\end{tabular}
&$n$ &\begin{tabular}{c}
$\varphi_1{+}\varphi_2$, $\pi_1{+}\varphi_1$, $\pi_1{+}\varphi_2$\\
for $n=3$\\
\hline
$\varphi_1$, $\pi_1{+}\varphi_1$, $\pi_1{+}\varphi_2$,\\
$\pi_2{+}\varphi_2,\dots,\pi_{k-2}{+}\varphi_{k-1}$,\\
$\pi_{k-1}{+}\pi_k{+}\varphi_{k-1},$\\
$\pi_{k-1}{+}\varphi_k$, $\pi_k{+}\varphi_k$\\
for $n=2k$\\
\hline
$\varphi_1$, $\pi_1{+}\varphi_1$, $\pi_1{+}\varphi_2$,\\
$\pi_2{+}\varphi_2, \dots,\pi_{k-1}{+}\varphi_{k-1}$,\\
$\pi_{k-1}{+}\varphi_k{+}\varphi_{k+1}$,\\
$\pi_k{+}\varphi_k$, $\pi_k{+}\varphi_{k+1}$\\
for $n=2k+1{\ge}5$
\end{tabular}
&$n\geqslant 3$
\\
\hline 3 &\begin{tabular}{l}
$\oSL_n{\times} \oSp_{2m}$ \\
$\qquad\supset\mathbb C^{\times}{\cdot} \oSL_{n-2}{\times}
\oSL_2{\times} \oSp_{2m-2}$
\end{tabular}
&\begin{tabular}{c}
\begin{picture}(40,23)
\put(2,2){\circle*{3}} \put(4,5){\line(1,2){5}}
\put(11,2){\circle*{3}} \put(11,5){\line(0,1){10}}
\put(11,18){\circle{3}}\put(13,15){\line(1,-2){5}}
\put(20,2){\circle*{3}} \put(22,5){\line(1,2){5}}
\put(29,18){\circle{3}}\put(31,15){\line(1,-2){5}}
\put(38,2){\circle*{3}}
\end{picture}
\end{tabular}
&$6^*$ &\begin{tabular}{c}
$(\pi_{n-2},2\chi_0),$\\
$(\varphi_2,0)$\,($m{\ge}2$),\\
$(\pi_{n-1}{+}\varphi_1,\chi_0)$,\\
$(\pi_1{+}\pi_{n-1},0)$\,($n{\ge}4$),\\
$(\pi_1{+}\varphi_1,-\chi_0),$\\
$(\pi_2,-2\chi_0)$\\
\end{tabular}
&\begin{tabular}{c}
$n\ge3$\\
$m\ge1$
\end{tabular}
\\
\hline 4 &\begin{tabular}{l}
$\oSL_n{\times} \oSp_{2m}$\\
$\qquad\quad\ \ \supset\oSL_{n-2}{\times} \oSL_2{\times}
\oSp_{2m-2}$
\end{tabular}
&\begin{tabular}{c}
\begin{picture}(40,23)
\put(2,2){\circle*{3}} \put(4,5){\line(1,2){5}}
\put(11,18){\circle{3}}\put(13,15){\line(1,-2){5}}
\put(20,2){\circle*{3}} \put(22,5){\line(1,2){5}}
\put(29,18){\circle{3}}\put(31,15){\line(1,-2){5}}
\put(38,2){\circle*{3}}
\end{picture}
\end{tabular}
&$6^*$ &\begin{tabular}{c}
$\pi_{n-2}$, $\varphi_2$\,($m{\ge}2$),\\
$\pi_{n-1}{+}\varphi_1$, $\pi_1{+}\pi_{n-1}$,\\
$\pi_1{+}\varphi_1$, $\pi_2$
\end{tabular}
&\begin{tabular}{c}
$n\ge5$ \\
$m\ge1$
\end{tabular}
\\
\hline 5
&\begin{tabular}{l}$\oSp_{2n}{\times} \oSp_{2m}$\\
$\qquad\quad\ \supset\oSp_{2n-2}{\times} \oSp_2{\times} \oSp_{2m-2}$
\end{tabular}
&\begin{tabular}{c}
\begin{picture}(40,23)
\put(2,2){\circle*{3}} \put(4,5){\line(1,2){5}}
\put(11,18){\circle{3}}\put(13,15){\line(1,-2){5}}
\put(20,2){\circle*{3}} \put(22,5){\line(1,2){5}}
\put(29,18){\circle{3}}\put(31,15){\line(1,-2){5}}
\put(38,2){\circle*{3}}
\end{picture}
\end{tabular}
&$3^*$ &\begin{tabular}{c} $\pi_2$\,($n{\ge}2$),
$\varphi_2$\,($m{\ge}2$),
\\
$\pi_1{+}\varphi_1$
\end{tabular}
&\begin{tabular}{c}
$n\ge1$\\
$m\ge1$
\end{tabular}
\\
\hline 6 &$\oSp_{2n}{\times} \oSp_4\supset\oSp_{2n-4}{\times}
\oSp_4$ &\begin{tabular}{c}
\begin{picture}(13,23)
\put(2,18){\circle{3}} \put(2,15){\line(0,-1){10}}
\put(2,2){\circle*{3}} \put(4,15){\line(1,-2){5}}
\put(11,18){\circle{3}}\put(11,15){\line(0,-1){10}}
\put(11,2){\circle*{3}}
\end{picture}
\end{tabular}
& $6^*$ &\begin{tabular}{c}
$\pi_1{+}\varphi_1$, $\pi_2{+}\varphi_2$,\\
$\pi_3{+}\varphi_1$, $\pi_4$\,($n{\ge}4$),\\
$\pi_2$, $\pi_1{+}\pi_3{+}\varphi_2$
\end{tabular}
&$n\ge3$
\\
\hline 7 &\begin{tabular}{l} $\oSp_{2n}{\times} \oSp_{2m}{\times}
\oSp_{2l}$\\
${\supset}\oSp_{2n-2}\!{\times} \oSp_{2m-2}\!{\times}
\oSp_{2l-2}\!{\times} \oSp_2$
\end{tabular}
&\begin{tabular}{c}
\begin{picture}(22,39)
\put(2,2){\circle*{3}} \put(2,5){\line(0,1){10}}
\put(2,18){\circle{3}} \put(4,21){\line(1,2){5}}
\put(11,34){\circle*{3}} \put(11,31){\line(0,-1){10}}
\put(11,18){\circle{3}} \put(11,15){\line(0,-1){10}}
\put(11,2){\circle*{3}} \put(13,31){\line(1,-2){5}}
\put(20,18){\circle{3}} \put(20,15){\line(0,-1){10}}
\put(20,2){\circle*{3}}
\end{picture}
\end{tabular}
& $6^*$ &\begin{tabular}{c} $\pi_2$\,($n{\ge}2$),
$\varphi_2$\,($m{\ge}2$),\\
$\psi_2$\,($l{\ge}2$), $\pi_1{+}\varphi_1$,\\
$\varphi_1{+}\psi_1$, $\pi_1{+}\psi_1$
\end{tabular}
&\begin{tabular}{c}
$n\ge1$\\
$m\ge1$\\
$l\ge1$
\end{tabular}
\\
\hline 8 &\begin{tabular}{l}
$\oSp_{2n}{\times} \oSp_4{\times} \oSp_{2m}$\\
$\ \ \supset\oSp_{2n-2}{\times} \oSp_2{\times} \oSp_2{\times}
\oSp_{2m-2}$
\end{tabular}
&\begin{tabular}{c}
\begin{picture}(58,23)
\put(2,2){\circle*{3}} \put(4,5){\line(1,2){5}}
\put(11,18){\circle{3}}\put(13,15){\line(1,-2){5}}
\put(20,2){\circle*{3}} \put(22,5){\line(1,2){5}}
\put(29,18){\circle{3}}\put(31,15){\line(1,-2){5}}
\put(38,2){\circle*{3}} \put(40,5){\line(1,2){5}}
\put(47,18){\circle{3}}\put(49,15){\line(1,-2){5}}
\put(56,2){\circle*{3}}
\end{picture}
\end{tabular}
& $6^*$ &\begin{tabular}{c}
$\varphi_2$, $\pi_1{+}\varphi_1$, $\varphi_1{+}\psi_1$,\\
$\pi_2$\,($n{\ge}2$), $\psi_2$\,($m{\ge}2$),\\
$\pi_1{+}\varphi_2{+}\psi_1$
\end{tabular}
&\begin{tabular}{c}
$n\ge1$\\
$m\ge1$
\end{tabular}
\\
\hline
\end{tabular}
}
\end{center}
\end{table}

We now formulate the main results of this paper. First, we compute
the semigroups $\widehat\Gamma(G/H)$ for all strictly irreducible
spherical pairs $(G,H)$, where $G$ is a simply connected non-simple
semisimple algebraic group, $H$ its connected reductive subgroup.
All such pairs except for symmetric pairs, which are to be discussed
later, are listed in~Table~\ref{tab1}. The indecomposable elements
of the corresponding extended weight semigroups are indicated in the
column `$\widehat\Gamma(G/H)$' of~Table~\ref{tab1}. Second, for each
space $G/H$ in Table~\ref{tab1} we find the weight functions in~$A$
that correspond to the indecomposable elements of
$\widehat\Gamma(G/H)$. These functions freely generate the
algebra~$A$.

Having known the semigroups $\widehat\Gamma(G/H)$ for the spaces in
Table~\ref{tab1} and taking into account Remark~\ref{rm1}, one can
obtain a description of the semigroups $\Gamma(G/H)$ for these
spaces. In~particular, $\widehat\Gamma(G/H)=\Gamma(G/H)$ for
spaces~2,~4--8 in~Table~\ref{tab1}.

Every simply connected strictly irreducible spherical homogeneous
space of a non-simple semisimple algebraic group $G$ that is
symmetric is isomorphic to one of the spaces of the form
$X(H)=(H\times H)/H$ (the subgroup~$H$ is embedded diagonally),
where~$H$ is simple and simply connected. For spherical homogeneous
spaces of this kind, the structure of the semigroup
$\widehat\Gamma(X(H))=\Gamma(X(H))$ is well known and follows, for
instance, from Theorem~\ref{th5} (see~\S\,\ref{ssec2.1} below) with
$L=K=H$. Namely, this semigroup is freely generated by the elements
$\pi_i+\varphi_i^*$, $i=1,\dots,\rk H$, where~$\pi_i$ and
$\varphi_i$ are the fundamental weights of the first and second
factors of $H\times H$, respectively. The asterisk denotes the
highest weight of the dual representation.

For computation of the extended weight semigroups, two different
approaches are used in this paper. The first one is applied to
spaces~1 and~2 in Table~\ref{tab1}, and the second to all the
others.

The paper is organized as follows. In~\S\,\ref{sec2} we formulate
and prove all the statements needed for our approaches for
computation of the extended weight semigroups. In~\S\,\ref{sec3} we
compute the semigroups $\widehat\Gamma(G/H)$ for each of the
homogeneous spaces $G/H$ appearing in Table~\ref{tab1}. We also find
the weight functions in~$A$ corresponding to the indecomposable
elements of the respective semigroups~$\widehat\Gamma(G/H)$.

Let us explain the notation used in~Table~\ref{tab1}. The first two
columns of this table are taken from Table~2 in~\cite{12} but are
represented in a form which is more convenient for our purpose. The
dot between the first and the second factor of~$H$ in row~3 denotes
their almost direct product, that is, these two factors have a
non-trivial (but finite) intersection.

The embedding diagram describes the embedding of~$H$ in~$G$ and
ought to be interpreted as follows. The white nodes correspond to
the factors of the group~$G$ and the black ones to those of~$H$. The
order of nodes is the same as that of the corresponding factors,
except for no.~7 where the upper black node corresponds to the last
factor of $H$. The factor of~$H$ corresponding to a black node~$v$
is diagonally embedded in the product of the factors of~$G$ that
correspond to the white nodes joined with~$v$.

The column `$\rk \widehat\Gamma(G/H)$' shows the rank of the
semigroup $\widehat\Gamma(G/H)$, that is, the number of
indecomposable elements of~$\widehat\Gamma(G/H)$. In this column, an
asterisk stands for the cases when, for several small values of the
parameters~$n$, $m$,~$l$, the rank of the extended weight semigroup
of the corresponding homogeneous space is less than the value
indicated in the table. The exact value of the rank for given values
of the parameters can be determined using the information in the
column `$\widehat\Gamma(G/H)$'.

The column `$\widehat\Gamma(G/H)$' contains a list of all
indecomposable elements $(\lambda,\chi)$ of the semigroup
$\widehat\Gamma(G/H)$. (If $\mathfrak{X}(H)=0$, then we write simply
$\lambda$ instead of $(\lambda,0)$.) These elements freely generate
it. If $(G,H)$ is a pair in Table~\ref{tab1} with
$\mathfrak{X}(H)\ne0$, then $H/H_0$ is a one-dimensional torus. Each
character $\chi\in\mathfrak{X}(H)$ is identified with some character
of this torus and, therefore, the characters of~$H$ are written in
the column `$\widehat\Gamma(G/H)$' as multiples of the
character~$\chi_0$, where $\chi_0$ is a fixed basis character
of~$H/H_0$. In each case, an explicit expression for~$\chi_0$ can be
found in~\S\,\ref{sec3}. Whenever an element~$(\lambda,\chi)$ is
followed by parenthesis containing an inequality for one of the
parameters~$n$, $m$,~$l$, the weight~$(\lambda,\chi)$ is contained
in the set of indecomposable elements of~$\widehat\Gamma(G/H)$ if
and only if the corresponding parameter satisfies that inequality.

\section*{Some notation and conventions}

If the group~$G$ (resp.~$H_0$) is a product of several factors, then
the $i$-th factor is denoted by~$G_i$ (resp.~$H_i$). We write $B_i$,
$U_i$, and $T_i$ respectively for the Borel subgroup, the maximal
unipotent subgroup, and the maximal torus of~$G_i$ such that
$B=\prod B_i$, $U=\prod U_i$, and $T=\prod T_i$. By~$\pi_i$,
$\varphi_i$, and~$\psi_i$ we denote the $i$-th fundamental weight of
the first, the second, and the third factor of~$G$, respectively.

Our numeration of fundamental weights of simply connected simple
algebraic groups is the same as in the book~\cite{13}.

For every semisimple group~$L$, we denote by~$V_\lambda(L)$ the
irreducible $L$-module with highest weight~$\lambda$. The weight
dual to the weight~$\lambda$ is denoted by~$\lambda^*$.

The identity element of any group is denoted by~$e$.

By~$\mathbb{C}^\times$ we denote the multiplicative group of the
field~$\mathbb{C}$.

If $P$ is a matrix, then the equation $P=(p_{ij})$ means that
$p_{ij}$ is the element in the~$i$-th row and the $j$-th column
of~$P$.

Given elements $a_1,\dots,a_n$ of a group, we write $S\langle
a_1,\dots\allowbreak\dots,a_n\rangle$ (resp. $\langle
a_1,\dots,a_n\rangle$) for the subsemigroup with identity (resp. the
subgroup) generated by $a_1,\dots,a_n$.

We denote by~$E_n$ the identity matrix of order~$n$ and by $F_n$ the
matrix of order~$n$ with ones on the antidiagonal and zeros
elsewhere.

The basis $e_1, \ldots, e_n$ of the space of the tautological linear
representation of the group $\oSp_{2m}$ is supposed to be chosen in
such a way that the matrix of the invariant non-degenerate
skew-symmetric bilinear form~$\omega_{2m}$ is
$$
\Omega_{2m}=
\begin{pmatrix}
0 & F_m
\\
-F_m & 0
\end{pmatrix}.
$$
With this choice of the basis, the Borel subgroup and the maximal
unipotent subgroup of~$\oSp_{2m}$ are represented by
upper-triangular matrices.

\section{Auxiliary results}
\label{sec2}

\subsection{}
\label{ssec2.1} Theorems~\ref{th3}--\ref{th5} are used to compute
the extended weight semigroups of spaces~1 and~2 in
Table~\ref{tab1}. Theorems~\ref{th3},~\ref{th4} describe known
branching rules for the groups $\oSL_{n+1}$, $\Spin_{n+1}$,
respectively (see original papers~\cite{14},~\cite{15} or a modern
exposition in~\cite{16}) stated in a form which is convenient for
our purpose.

\begin{theorem}[{\rm branching rule for the group $\oSL_{n+1}$}]
\label{th3} Suppose
$\lambda=c_1\pi_1+\dotsb\allowbreak\dots+c_n\pi_n$ is a dominant
weight of $\oSL_{n+1}$, where $c_i\ge0$. Then the restriction of the
irreducible representation of\, $\oSL_{n+1}$ with highest
weight~$\lambda$ to the subgroup $\oSL_n\subset\oSL_{n+1}$ has the
form $\bigoplus \limits_{\mu\in M(\lambda)}V_\mu(\oSL_n)$, where the
set~$M(\lambda)$ consists of dominant weights $\mu$ of\, $\oSL_n$
\textup{(}possibly with multiplicities\textup{)} that can be
obtained from~$\lambda$ by simultaneously replacing all summand of
the form $c_i\pi_i$ by $a_i\pi_{i-1}+b_i\pi_i$, where $a_i,b_i\ge0$
and $a_i+b_i=c_i$. At that, we put $\pi_0=0$ and
$\pi_n\big|_{\oSL_n}=0$.
\end{theorem}

Before we formulate the next theorem, let us note that every
dominant weight~$\lambda$ of the group $\Spin_{2k+2}$ is uniquely
expressible in the form $\lambda=c_1\pi_1
+\dotsb\allowbreak\dots+c_{k+1}\pi_{k+1}+d(\pi_k+\pi_{k+1})$, where
$c_i,d\ge0$ and at least one of the numbers~$c_k$, $c_{k+1}$ is
zero.

\begin{theorem}[{\rm branching rules for the group $\Spin_{n+1}$}]
\label{th4} {\rm a)} Suppose $\lambda\,{=}\,
c_1\pi_1+\nobreak\dotsb\allowbreak\dots+c_k\pi_k$ is a dominant
weight of\, $\Spin_{2k+1}$, $k\ge2$, where $c_i\ge0$. Then the
restriction of the irreducible representation of\; $\Spin_{2k+1}$
with highest weight~$\lambda$ to the subgroup $\Spin_{2k}$ has the
form $\bigoplus \limits_{\mu\in M(\lambda)}V_\mu(\Spin_{2k})$, where
the set~$M(\lambda)$ consists of dominant weights~$\mu$ of
$\Spin_{2k}$ that can be obtained from~$\lambda$ by simultaneously
replacing all summands of the form~$c_i\pi_i$ \textup{(}for $i\ne
k-1$\textup{)} by $a_i\pi_{i-1}+b_i\pi_i$ and the summand
$c_{k-1}\pi_{k-1}$ by $a_{k-1}\pi_{k-2}+b_{k-1}(\pi_{k-1}+\pi_k)$,
where $a_i,b_i\ge0$ and $a_i+b_i=c_i$. At that, we put $\pi_0=0$.

{\rm b)} Suppose
$\lambda=c_1\pi_1+\cdots+c_{k+1}\pi_{k+1}+d(\pi_k+\pi_{k+1})$ is a
dominant weight of\, $\Spin_{2k+2}$, $k\ge1$, where $c_i,d\ge0$ and
at least one of the numbers~$c_k$, $c_{k+1}$ is zero. Then the
restriction of the irreducible representation of\, $\Spin_{2k+2}$
with highest weight~$\lambda$ to the subgroup $\Spin_{2k+1}$ has the
form $\bigoplus \limits_{\mu\in M(\lambda)}V_\mu(\Spin_{2k+1})$,
where the set~$M(\lambda)$ consists of dominant weights~$\mu$ of
$\Spin_{2k+1}$ that can be obtained from~$\lambda$ by simultaneously
replacing all summands of the form~$c_i\pi_i$ \textup{(}for $i<
k$\textup{)} by $a_i\pi_{i-1}+b_i\pi_i$, where $a_i,b_i\ge0$ and
$a_i+b_i=c_i$, all summands of the form~$c_i\pi_i$ \textup{(}for
$i=k,k+1$\textup{)} by $c_i\pi_k$, and the summand
$d(\pi_k+\pi_{k+1})$ by $a\pi_{k-1}+2b\pi_k$, where $a,b\ge0$ and
$a+b=d$. At that, we put $\pi_0=0$.
\end{theorem}

\begin{theorem}
\label{th5} Suppose~$L$, $K$ are connected semisimple algebraic
groups and $L\subset K$. Let $m_{\lambda,\mu}$ be the multiplicity
with which the irreducible representation of\,~$L$ with highest
weight~$\mu$ occurs in the irreducible representation of~$K$ with
highest weight~$\lambda$. Consider the group $G=L\times K$ and its
subgroup $H\simeq L$ embedded in~$G$ diagonally. Then there is an
isomorphism of $G$-modules $\mathbb{C}[G]^H\simeq \bigoplus
\limits_{\lambda,\mu} m_{\lambda,\mu} V_{\mu+\lambda^*}(G)$
\textup{(}$G$ acts on~$\mathbb{C}[G]^H$ by left
multiplication\textup{)}, where~$\lambda$, $\mu$ run over all
dominant weights of\,~$K$, $L$, respectively.
\end{theorem}

\begin{proof}
Consider the space~$\mathbb{C}[K]$ as a $G$-module on which~$L$
and~$K$ act by left and right multiplication respectively. There is
an isomorphism of algebras and $G$-modules $\varphi\colon
\mathbb{C}[G]^H \to \mathbb{C}[K]$ given by $(\varphi f)(x)=f(e,x)$.
Further, it is well-known that the space~$\mathbb{C}[K]$, regarded
as a $(K\times K)$-module with respect to the actions by left and
right multiplication, is isomorphic to the direct sum $\bigoplus
\limits_\lambda V_\lambda(K)^*\otimes V_\lambda(K)$, where~$\lambda$
runs over all dominant weights of~$K$. Restricting the action of
~$K$ by right multiplication to~$L$ and taking into account the
relations $V_\lambda(K)^*\simeq V_{\lambda^*}(K)$ and
$V_{\lambda^*}(K)\otimes V_\mu(L) \simeq V_{\lambda^*+\mu}(K\times
L) \simeq V_{\mu+\lambda^*}(G)$, we get the desired result.
\end{proof}

\subsection{}
\label{ssec2.2} The results in this subsection are used for
computation of the extended weight semigroups of spaces~3--8 in
Table~\ref{tab1}.

\begin{lemma}
\label{lem1} Suppose a group~$L$ acts on an irreducible algebraic
variety~$X$. Suppose there is a closed subset $Y\subset X$, which is
called a section, and an open subset $M\subset X$ such that the
orbit of any point in~$M$ meets~$Y$. Then the restriction
homomorphism $\rho\colon \mathbb{C}[X]^L\to\mathbb{C}[Y]$ is
injective.
\end{lemma}

\begin{proof}
Assume that $\rho(f)=0$ for some function $f\in\mathbb{C}[X]^L$.
satisfies . Then $f\big|_M=0$ because invariant functions are
constant along orbits. This implies $f\equiv 0$.
\end{proof}

\begin{theorem}
\label{th6} Suppose the group $\oSp_{2m-2k}$, $k\ge 1$, is embedded
in $\oSp_{2m}$ as the central $(2m-2k)\times(2m-2k)$ block and acts
on $\oSp_{2m}$ by right multiplication. Then the algebra of
invariants of this action coincides with the subalgebra of the
algebra $\mathbb{C}[\oSp_{2m}]$ generated by the matrix entries of
the first~$k$ and last~$k$ columns.
\end{theorem}

\begin{proof}
Let $V$ be the space of the tautological representation
of~$\oSp_{2m}$. Consider the natural action of~$\oSp_{2m}$ on the
space
$$
W=\underbrace{V\oplus \cdots \oplus V}_k \oplus \underbrace{V\oplus
\cdots \oplus V}_k.
$$
The subgroup $\oSp_{2m-2k}$ in the hypothesis of the theorem is the
stabilizer of the vector $w=(e_1,e_2,\dots,e_k,
e_{2m-k+1},e_{2m-k+2},\dots,e_{2m})$ under this action. The orbit
of~$w$ in~$W$ is isomorphic to the quotient space
$\oSp_{2m}/\oSp_{2m-2k}$ and this isomorphism induces the
correspondence between regular functions on this orbit and the
required invariants. The orbit is closed since it consists of the
sets of vectors $(v_1,v_2,\dots,v_{2k})$ whose Gram matrix with
respect to the form~$\omega_{2m}$ is~$\Omega_{2k}$. Hence, the
algebra of regular functions on the orbit is generated by the
restrictions of the coordinates of the ambient space~$W$. These
coordinates correspond to the matrix entries of the first~$k$ and
the last~$k$ columns of a matrix in~$\oSp_{2m}$.
\end{proof}

The next three lemmas are technical. The proofs of the first two of
them are obtained by direct computation.

\begin{lemma}
\label{lem2} For every non-degenerate matrix~$P$ of order~$m$, the
matrix
\begin{equation}
\label{eq1}
\begin{pmatrix}
P & 0
\\
0 & (P^{-1})^{\#}
\end{pmatrix}
\end{equation}
of order~$2m$ is symplectic. \textup{(}Here the symbol~$\#$ denotes
the transpose of a matrix with respect to the
antidiagonal.\textup{)}
\end{lemma}

\begin{lemma}
The matrix
\begin{equation}
\label{eq2}
\begin{pmatrix}
E_m & C
\\
0 & E_m
\end{pmatrix}
\end{equation}
of order~$2m$ is symplectic if and only if the matrix~$C$ of
order~$m$ is symmetric with respect to the antidiagonal.
\end{lemma}

\begin{lemma}
\label{lem4} Let $P=(p_{ij})$ be a $2m\times 2$ matrix, $m\ge 2$,
and~$P_1$, $P_2$ its first and second columns, respectively. Suppose
$p_{2m,1}\ne0$, $\Delta=p_{2m-1,1}p_{2m,2}-p_{2m-1,2}p_{2m,1}\ne0$,
and $P_1^\top\Omega_{2m}P_2=1$ \textup{(}the symbol~$\top$ denotes
the transposed matrix\textup{)}. Then there are upper unitriangular
matrices $u_1, u_2\in \oSp_{2m}$ such that
$$
u_1P=\begin{pmatrix} 0 &-\dfrac{1}{p_{2m,1}}
\\
0 & 0
\\
\hdots & \hdots
\\
0 & -\dfrac{\Delta}{p_{2m,1}}
\\
\ p_{2m,1} & \ \ p_{2m,2}\end{pmatrix}, \qquad u_2P=\begin{pmatrix}
0 & 0
\\
-\dfrac{p_{2m,1}}{\Delta} & 0
\\
\hdots & \hdots
\\
0 & -\dfrac{\Delta}{p_{2m,1}}
\\
\ p_{2m,1} & \ \ p_{2m,2}
\end{pmatrix}.
$$
\textup{(}The dots stand for zero entries.\textup{)}
\end{lemma}

\begin{proof}
Multiplying~$P$ on the left by an appropriate upper unitriangular
matrix of type~\eqref{eq1}, we obtain a matrix~$P'$ whose lower half
contains only three non-zero elements:~$p_{2m,1}$, $p_{2m,2}$, and
$-\Delta/p_{2m,1}$ (as in the matrices~$u_1P$ and~$u_2P$ appearing
in the assertion of the lemma). Then, multiplying~$P'$ on the left
by an appropriate matrix of type~\eqref{eq2}, we obtain one of the
required matrices.
\end{proof}

The following theorem is used at the final stage of the argument in
all cases.

\begin{theorem}
\label{th7} Suppose $G$ is simply connected and $H\subset G$ is a
connected spherical subgroup. Suppose non-zero functions
$f_1,\dots,f_n\in A$ satisfy the following conditions:

{\rm a)} $f_i\in A_{\lambda_i,\chi_i}$ for $i=1,\dots,n$, where
$\lambda_i\in\mathfrak{X}_+(B)$, $\chi_i\in\mathfrak{X}(H)$, and the
weights $(\lambda_1,\chi_1),\dots,(\lambda_n,\chi_n)$ are linearly
independent;

{\rm b)} there is an inclusion $A\subset
\mathbb{C}[f_1,\dots,f_n,f_1^{-1},\dots,f_k^{-1}]$ for some
$k\le\nobreak n$.

Put $Z=\langle(\lambda_1,\chi_1),\dots,(\lambda_k,\chi_k)\rangle+
S\langle(\lambda_{k+1},\chi_{k+1}),\dots,(\lambda_n,\chi_n)\rangle$.
Then:

{\rm 1)} if for each expression
\begin{equation}
\label{eq3}
(\lambda_i,\chi_i)=(\mu_1,\sigma_1)+(\mu_2,\sigma_2),\qquad
\mu_1,\mu_2 \in \mathfrak{X}_+(B)\setminus\{0\},\quad
\sigma_1,\sigma_2 \in\mathfrak{X}(H),
\end{equation}
at least one of the weights $(\mu_1,\sigma_1)$, $(\mu_2,\sigma_2)$
is not contained in~$Z$, then the element~$f_i$ is irreducible
in~$A$;

{\rm 2)} if
$(\lambda_i,\chi_i)=(\lambda_j,\chi_j)+(\lambda_i-\lambda_j,\chi_i-\chi_j)$,
$i\ne j$, is the unique expression of the weight
$(\lambda_i,\chi_i)$ in the form~\eqref{eq3} such that both summands
belong to~$Z$, and if~$f_i$ is not divisible by~$f_j$ in~$A$, then
the element~$f_i$ is irreducible in~$A$;

{\rm 3)} if $f_i$ is irreducible in~$A$ for $i=1,\dots,n$, then
$A=\mathbb{C}[f_1,\dots,f_n]$ and $\widehat\Gamma(G/H)=
S\langle(\lambda_1,\chi_1),\dots,(\lambda_n,\chi_n)\rangle$.
\end{theorem}

\begin{proof}
Let us prove~1),~2). Assume that~$f_i$ is reducible in~$A$. Then
$f_i=F_1F_2$ where for $j=1,2$ the element $F_j\in A$ is not
invertible and $F_j\in A_{\xi_j,\eta_j}$ for some
$\xi_j\in\mathfrak{X}_+(B)$, $\eta_j\in\mathfrak{X}(H)$. We have
$\xi_j\ne0$, $j=1,2$, because otherwise~$F_j$ would be a constant.
It follows from~b) that each of the functions~$F_1$,~$F_2$ is
expressible as an irreducible fraction of the form
$$
F_j=\frac {h_j(f_1,\dots,f_n)}{f_1^{\beta_{j1}}\cdots
f_k^{\beta_{jk}}}\,,
$$
where $\beta_{j1},\dots,\beta_{jk}\ge0$ and $h_j$ are polynomials
in~$n$ variables. Using~a), we obtain $h_j(f_1,\dots,f_n)= c_j
f_1^{\alpha_{j1}}\cdots f_n^{\alpha_{jn}}$, where $c_j\ne0$ and
$\alpha_{j1},\dots,\alpha_{jn}\ge0$. Then $(\mu_j,\sigma_j) \in Z$
for $j=1,2$. In the hypothesis of~1), we have already come to a
contradiction. In the hypothesis of\,~2), we see that one of the
functions~$F_1$,~$F_2$ has weight $(\lambda_j,\chi_j)$ and,
therefore, is proportional to~$f_j$ since $\dim
A_{\lambda_j,\chi_j}=1$. Thus, $f_i$ is divisible by $f_j$, a
contradiction.

We now prove~3). Suppose $f\in A_{\lambda,\chi}$ for some
$\lambda\in\mathfrak{X}_+(B)$, $\chi\in\mathfrak{X}(H)$. Arguing as
in the proof of~1), we see that~$f$ is expressible as an irreducible
fraction of the form $\frac{f_1^{\alpha_1}\cdots
f_n^{\alpha_n}}{f_1^{\beta_1}\cdots f_k^{\beta_k}}$\,, where
$\beta_1,\dots,\beta_k\ge0$ and $\alpha_1,\dots,\alpha_n\ge0$.
Since~$A$ is factorial (Theorem~\ref{th1}) and all the
elements~$f_i$ are irreducible, it follows that the numerator of
this fraction is divisible by its denominator, whence
$\beta_1=\dots=\beta_k=0$. Therefore, $f=f_1^{\alpha_1}\cdots
f_n^{\alpha_n}$ and we obtain the required result.
\end{proof}

\section{Computation of the extended weight semigroups}
\label{sec3}

In this section, the cases are numbered in accordance with the
numbers in~Table~\ref{tab1}. Except in Case~2 we use the following
convention. For each factor $G_i\subset G$ (all of them are of type
$\oSL$ or $\oSp$), the subgroups~$B_i$, $U_i$, and~$T_i$ consist of
all upper triangular, upper unitriangular, and diagonal matrices,
respectively, contained in~$G_i$.

\subsection{}
\label{ssec3.1} At first, we compute the semigroups
$\widehat\Gamma(G/H)$ for spaces~1,~2 in Table~\ref{tab1}.

\textsl{Case}~1. $G=\oSL_n\times \oSL_{n+1}$, $H=\oSL_n\times
\mathbb{C}^{\times}$, $H$ being embedded in~$G$ in such a way that
the image in~$G$ of a pair $(P,t)\in H$ is the pair $(P,P')\subset
G$, where $P'=\varphi(P,t)=\bigl(\begin{smallmatrix} Pt & 0 \\ 0 &
t^{-n}
\end{smallmatrix}\bigr)$. Further, $H_0=\oSL_n\times\{e\}\subset H$.
The basis character $\chi_0\in\mathfrak{X}(H)$ takes each element
$(P,t)\in H$ to~$t$.

Let us present $2n$ functions in~$A$ that are weight functions with
respect to $B\times H$. Given $(P,Q)\in G$, we put $R=Q\widetilde
P^{-1}$, where $\widetilde P=\varphi(P,e)$. We denote by~$\Delta_i$,
$i=1,\dots,n$, the minor of~$R$ corresponding to the last~$i$ rows
and first~$i$ columns, and by~$\delta_i$, $i=1,\dots,n$, the minor
of~$R$ corresponding to the last $i$ rows and columns
$n+1,1,2,\dots,i-1$. The~$2n$ functions $\Delta_1,\dots,\Delta_n$,
$\delta_1,\dots,\delta_n$ all belong to~$A$ and are weight functions
with respect to $B\times H$. The weights of
$\Delta_1,\Delta_2,\dots,\Delta_{n-1},\Delta_n$ are
$$
(\pi_1+\varphi_n,\chi_0),\ (\pi_2+\varphi_{n-1},2\chi_0),\ \dots,\
(\pi_{n-1}+\varphi_2,(n-1)\chi_0),\ (\varphi_1,n\chi_0),
$$
respectively, the weights of
$\delta_1,\delta_2,\delta_3,\dots,\delta_n$ are
$$
(\varphi_n,-n\chi_0),\ (\pi_1+\varphi_{n-1},-(n-1)\chi_0),\
(\pi_2+\varphi_{n-2},-(n-2)\chi_0),\ \dots,\
(\pi_{n-1}+\varphi_1,-\chi_0),
$$
respectively. Moreover, the~$2n$ weights of the functions
$\Delta_1,\dots,\Delta_n$, $\delta_1,\dots,\delta_n$ are linearly
independent, whence $\rk\widehat\Gamma(G/H)\ge2n$. Since $\rk
G=2n-1$ and $\rk H/H_0=1$, it follows that $\rk\widehat\Gamma(G/H)
\le 2n$. Hence $\rk\widehat\Gamma(G/H)=2n$.

Applying Theorem~\ref{th5} for $L=\oSL_n$, $K=\oSL_{n+1}$ and using
Theorem~\ref{th3}, we determine the spectrum of the representation
of~$G$ on the space $\mathbb{C}[G]^{H_0}$. (Thereby we determine the
semigroup $\Gamma(G/H_0)$.) In particular, we obtain the following
two facts. First, this spectrum contains the irreducible $G$-modules
with highest weights
\begin{equation}
\label{eq4} \pi_1+\varphi_n,\ \pi_2+\varphi_{n-1},\ \dots,\
\pi_{n-1}+\varphi_2,\ \varphi_1, \qquad \varphi_n,\
\pi_1+\varphi_{n-1},\ \pi_2+\varphi_{n-2},\ \dots,\
\pi_{n-1}+\varphi_1
\end{equation}
($2n$ weights in total), each of multiplicity~$1$. Second, any
non-zero weight in $\Gamma(G/H_0)$ contains at least one summand of
the form~$\varphi_i$. It follows that each of the
weights~\eqref{eq4} is indecomposable in $\Gamma(G/H_0)$. We note
that the set of~$2n$ weights~\eqref{eq4} is the image under the
map~$\pi$ (see Remark~\ref{rm2}) of the set of~$2n$ weights with
respect to $B\times H$ that correspond to the functions
$\Delta_1,\dots,\Delta_n$, $\delta_1,\dots,\delta_n$. Hence, by
Remark~\ref{rm2}, the latter~$2n$ weights are indecomposable in
$\widehat\Gamma(G/H)$. Since $\rk\widehat\Gamma(G/H)=2n$, it follows
that the semigroup $\widehat\Gamma(G/H)$, which is free, is
generated by the weights with respect to $B\times H$ of the
functions $\Delta_1,\dots,\Delta_n$, $\delta_1,\dots,\delta_n$.
Therefore
$A=\mathbb{C}[\Delta_1,\dots,\Delta_n,\delta_1,\dots,\delta_n]$.

\textsl{Case}~2. $G=\Spin_n \times \Spin_{n+1}$, $H=H_0=\Spin_n$.
Since $\mathfrak{X}(H)=0$, we get $\widehat\Gamma(G/H)=\Gamma(G/H)$
and a description of the semigroup~$\Gamma(G/H)$ follows from
Theorems~\ref{th5},~\ref{th4}. Namely, a direct check shows that,
for each $n\ge3$, the weights in the column `$\widehat\Gamma(G/H)$'
of Table~\ref{tab1} (there are exactly~$n$ of them) lie in the
semigroup~$\widehat\Gamma(G/H)$ and linearly independent, therefore
$\rk \Gamma(G/H)\ge n$. Since $\rk\Gamma(G/H)\le\rk G=n$, we get
$\rk\Gamma(G/H)=n$. Further, it is easy to see that every non-zero
element of $\Gamma(G/H)$ contains one of the elements~$\varphi_i$ as
a summand, whence all these weights are indecomposable except for
the weight $\pi_{k-1}+\varphi_k+\varphi_{k+1}$ for $n=2k+1$. The
last weight is indecomposable in~$\Gamma(G/H)$ because none of the
weights~$\varphi_k$, $\varphi_{k+1}$ is contained in~$\Gamma(G/H)$.
Thus, for each $n\ge3$ we have found~$n$ indecomposable linearly
independent weights in the semigroup $\Gamma(G/H)$. Since
$\rk\Gamma(G/H)=n$, these~$n$ weights freely generate the semigroup.

We now present weight functions generating the algebra $A(G'/H')$
where $G'/H'$ is a homogeneous space locally isomorphic to~$G/H$.
Namely, we consider the group $G'=\SO_n\times\SO_{n+1}$ and its
subgroup $H'=\SO_n$ embedded in~$G'$ diagonally. The covering
homomorphism of groups $\psi\colon G\to G'$ induces the morphism
$\psi_H\colon G/H\to G'/H'$, which is a two-sheeted covering.
Therefore there is an embedding of algebras
$$
\psi_H^*\colon \mathbb{C}[G'/H']\hookrightarrow \mathbb{C}[G/H],
$$
at that, $\psi_H^*(A(G'/H'))\subset A(G/H)$.

For each~$m$ we choose a basis~$\{e_i\}$ in the space~$V_m$ of the
tautological representation of~$\SO_m$ such that the matrix of the
invariant non-degenerate symmetric bilinear form is~$F_m$. Then all
upper-triangular and diagonal matrices in~$\SO_m$ form a Borel
subgroup~$\widetilde B_m$ and a maximal torus~$\widetilde T_m$,
respectively. We shall consider weights of irreducible
representations of $\SO_m$ with respect to~$\widetilde B_m$
and~$\widetilde T_m$. We fix the embedding $\tau_m\colon \SO_m
\hookrightarrow \SO_{m+1}$ such that its image is the stabilizer of
the vector $e_{\frac{m}2+1}$ for even~$m$ and the vector
$e_{\frac{m+3}2}-e_{\frac{m+1}2}$ for odd~$m$.

We fix the embedding~$H'$ in~$G'$ sending every matrix $P\in H'$ to
$(P,\tau_n(P))$. Suppose $(P,Q)\in G'$ and put $R=Q\tau_n(P)^{-1}$.
Now we present the generators of $A(G'/H')$ for each~$n$.

If $n=2k$, then we consider the following functions on~$R$:
$\Delta_i$, $i=1,\dots,k$, is the minor corresponding to the
last~$i$ rows and first~$i$ columns;~$\delta_i$,~$i=1,\dots,k$, is
the minor corresponding to the last~$i$ rows and columns
$k+\nobreak1,1,\dots,i-1$; $\Phi$ is the minor corresponding to the
last~$k$ rows and columns $k+2,1,\dots,k-1$. We have
$\delta_k^2=-2\Delta_k\Phi$. All these functions are weight
functions with respect to $(\widetilde B_n\times \widetilde
B_{n+1})\times H'$ and generate $A(G'/H')$. The weight of~$\Delta_i$
is $\pi_i+\varphi_i$ for $i\le k-2$, $\pi_{k-1}+\pi_k+\varphi_{k-1}$
for $i=k-1$, and $2(\pi_k+\varphi_k)$ for $i=k$. The weight
of~$\delta_i$ is $\pi_{i-1}+\varphi_i$ for $i\le k-1$ (we put
$\pi_0=0$) and $\pi_{k-1}+\pi_k+2\varphi_k$ for $i=k$. The weight
of~$\Phi$ is $2(\pi_{k-1}+\varphi_k)$. The algebra $A(G/H)$ also
contains functions~$\Delta$ and~$D$ such that $\Delta^2
=\psi_H^*(\Delta_k)$, $D^2=\psi_H^*(\Phi)$, and $\sqrt{-2}\,\Delta
D=\psi_H^*(\delta_k)$. Their weights are $\pi_k+\varphi_k$ and
$\pi_{k-1}+\varphi_k$, respectively. The functions
$\psi_H^*(\Delta_i)$ and $\psi_H^*(\delta_j)$, $i,j\le k-1$, along
with the functions~$\Delta$ and~$D$ correspond to the indecomposable
elements of $\widehat\Gamma(G/H)$ and generate~$A(G/H)$.

If $n=2k+1$, then we consider the following functions on~$R$:
$\Delta_i$, $i=1,\dots,k$, is the minor corresponding to the
last~$i$ rows and first~$i$ columns; $\delta_i$, $i=1,\dots,k+1$, is
the difference of two minors, the first corresponding to the
last~$i$ rows and columns $k+1,1,\dots,i-1$, and the second
corresponding to the last~$i$ rows and columns $k+2,1,\dots,i-1$;
$\Phi$ is the minor corresponding to rows $1,\dots,k,k+2$ (counting
from the bottom) and columns $1,\dots,k,2\bigl[\frac{k}2\bigr]+2$.
We have $\Delta_k^2=(-1)^{k+1} \delta_{k+1}\Phi$. All these
functions are weight functions with respect to $(\widetilde
B_n\times \widetilde B_{n+1})\times H'$ and generate $A(G'/H')$. The
weight of~$\Delta_i$ is $\pi_i+\varphi_i$ for $i\le k-1$ and
$2\pi_k+\varphi_k+\varphi_{k+1}$~for $i=k$. The weight of~$\delta_i$
is $\pi_{i-1}+\varphi_i$ for $i\le k-1$,
$\pi_{k-1}+\varphi_k+\varphi_{k+1}$ for~$i=k$, and
$2(\pi_k+\varphi_{k+1})$ for $i=k+1$ (we put $\pi_0=0$). The weight
of~$\Phi$ is $2(\pi_k+\varphi_k)$. The algebra $A(G/H)$ also
contains functions~$\delta$ and~$D$ such that
$\delta^2=\psi_H^*(\delta_{k+1})$, $D^2=\psi_H^*(\Phi)$, and
$\sqrt{(-1)^{k+1}}\,\delta D=\Delta_k$. Their weights are
$\pi_k+\varphi_{k+1}$ and $\pi_k+\varphi_k$, respectively. The
functions $\psi_H^*(\Delta_i)$, $i\le k-1$, $\psi_H^*(\delta_j)$,
$j\le k$, $\delta$, and~$D$ correspond to the indecomposable
elements of $\widehat\Gamma(G/H)$ and generate $A(G/H)$.

\subsection{}
\label{ssec3.2} We now proceed to computing the extended weight
semigroups of spaces~3--8 in Table~\ref{tab1}. First we describe the
general method.

In each of the cases considered below we search for the algebra~$A$.
The functions in this algebra satisfy $f(g)=f(u^{-1}gh)$ for all
$g\in G$, $u\in U$, $h\in H_0$. To find such functions, we multiply
an arbitrary element~$g$ of some dense open subset $M\subset G$ by
appropriate elements in~$U$ and~$H_0$ so as to obtain an element of
`canonical' form. A canonical form for elements of~$M$ is specified
by the condition that some of the matrix entries equal zero, some
others equal one, and some of the remaining entries equal minus one.
The set~$Y$ of elements of the \textit{whole} group~$G$ (not only
in~$M$) satisfying these restrictions is closed in~$G$ and serves as
a section in the sense of Lemma~\ref{lem1}. By that lemma, $A$ is
contained in the algebra~$\mathbb{C}[Y]$.

In all the cases we first present a set of functions
$f_1,\dots,f_p\in A$ that are weight functions with respect to
$B\times H$ and then we use the canonical form to prove that these
functions generate~$A$. (In almost all cases the functions
$f_1,\dots,f_p$ naturally arise when reducing to canonical form.)
The set~$M$ is determined by the condition that some of the
functions~$f_1,\dots,f_p$ do not vanish.

If one of the factors of~$G$ is $\oSp_{2m}$ and one of the factors
of~$H_0$ is $\oSp_{2m-2k}$ embedded only in~$\oSp_{2m}$ (as the
central $(2m-2k)\times(2m-2k)$ block), then by Theorem~\ref{th6}
every function in~$A$ is independent of the matrix entries of the
factor $\oSp_{2m}$ located in columns
$$
k+1,\ k+2,\ \dots,\ 2m-k-1,\ 2m-k.
$$
Therefore, when reducing to canonical form we may care not about
transformations of these entries under the actions of~$U$ and~$H_0$.
In this connection, firstly, we do not consider anymore the action
of the factor $\oSp_{2m-2k}\subset H_0$ since it transforms in a
non-trivial way only columns of~$\oSp_{2m}$ mentioned above.
Secondly, it is sufficient for our purpose to reduce to canonical
form not the whole of a matrix in~$\oSp_{2m}$ but only its first~$k$
and last~$k$ columns. In other words, it suffices to impose
restrictions of the form $g_\alpha=c$, where $c\in\{0,1,-1\}$,
defining a canonical form of a matrix in $\oSp_{2m}$ only on the
matrix entries~$g_\alpha$ of the factor $\oSp_{2m}$ that are located
in the~$2k$ columns indicated above. Therefore, when formally
considering matrices $Q\in\oSp_{2m}$, we actually deal only with
their submatrices~$\overline Q$ consisting of the first~$k$ and the
last~$k$ columns of~$Q$, and it is~$\overline Q$ that is reduced to
canonical form. This can be interpreted as follows: the factor
$\oSp_{2m}$ of~$G$ is replaced by the quotient space
$\oSp_{2m}/\oSp_{2m-2k}$ on which the actions of~$U$ and the
remaining factors of~$H_0$ are preserved. As we see from the proof
of Theorem~\ref{th6}, this quotient space can be thought of as the
set of $2m\times 2k$ matrices whose columns satisfy the same
relations as the first~$k$ and the last~$k$ columns of a matrix
in~$\oSp_{2m}$.

It always turns out that the matrix entries of the canonical form of
an element $g\in M$ on which the functions in~$A$ can depend are
rational functions (more precisely, Laurent polynomials) in the
values of $f_1,\dots,f_p$ at the point~$g$. We denote these rational
functions, which are obviously invariant with respect to~$U$
and~$H_0$, by $r_1(f_1,\dots,f_p)$, $\dots$, $r_q(f_1,\dots,f_p)$.
Since regular functions on the section are generated by the
restrictions of the coordinate functions on~$G$, it follows that
every function $f\in A$ is a polynomial in the functions
$r_1,\dots,r_q$, that is,
$$
f(g)=F\bigl(r_1(f_1(g),\dots,f_p(g)),\dots,
r_q(f_1(g),\dots,f_p(g))\bigr)
$$
for every $g\in M$, where $F(x_1,\dots,x_q)$ is a polynomial. Thus,
there is an inclusion $A \subset \tilde A=
\mathbb{C}[r_1(f_1,\dots,f_p),\dots,r_q(f_1,\dots,f_p)]$.

Since the algebra~$\tilde A$ consists of rational functions that are
invariant under~$U$ and~$H_0$, it follows that $A=\tilde A \cap
\mathbb{C}[G]$. Therefore our next step aims at extracting regular
functions from~$\tilde A$. This is carried out using
Theorem~\ref{th7}. Namely, it follows from assertions~1) and~2) (the
latter one is used only in Case~3 for $n=3$ and Case~8) of
Theorem~\ref{th7} that each of the functions $f_1,\dots,f_p$ is
irreducible in~$A$. Then assertion~3) yields that
$A=\mathbb{C}[f_1,\dots,f_p]$. A description of the semigroup
$\widehat\Gamma(G/H)$ also follows from assertion~3) of
Theorem~\ref{th7}.

In all the cases the conditions~a) and~b) of Theorem~\ref{th7} are
verified directly, therefore we do not even mention that except in
Cases~3,~4. We check the hypothesis of assertion~1) of this theorem
only in Cases~3,~4 since it is verified similarly in all the
remaining cases.

We now proceed to consideration of all the cases. Our argument
follows the plan discussed above, which is therefore used without
extra explanation.

\textsl{Cases}~3,~4. It is convenient to consider the pairs in
rows~3,~4 of Table~\ref{tab1} together. In both cases,
$G=\oSL_n\times \oSp_{2m}$, in Case~3 we have
$H=\mathbb{C}^{\times}\cdot \oSL_{n-2}\times \oSL_2\times
\oSp_{2m-2}$, in Case~4 we have $H=\oSL_{n-2}\times \oSL_2\times
\oSp_{2m-2}$. Case~4 is considered only for $n\ge5$, otherwise the
space $G/H$ is not spherical. The embedding of~$H$ in~$G$ is as
follows. The factor $\oSL_{n-2}$ is embedded in the factor $\oSL_n$
of~$G$ as the upper left $(n-2)\times (n-2)$ block. The factor
$\oSL_2$ is diagonally embedded in~$G$ as the lower right $2\times
2$ block in~$\oSL_n$ and as the $2\times 2$ block corresponding to
the first and last rows and columns of the factor $\oSp_{2m}$. The
factor $\oSp_{2m-2}$ is embedded in~$\oSp_{2m}$ as the central
$(2m-2)\times (2m-2)$ block. In Case~3 the torus
$\mathbb{C}^\times\subset H$ is embedded in~$\oSL_n$ as
$E_{n-2}t^{-2}\oplus E_2t^{n-2}$ for odd~$n$ and as
$E_{n-2}t^{-1}\oplus E_2t^{\frac{n-2}2}$ for even~$n$. The
group~$H_0$ is the same in both cases and equals $\oSL_{n-2}\times
\oSL_2\times \oSp_{2m-2}$. In~Case~3 the basis character
$\chi_0\in\mathfrak{X}(H)$ acts on the torus
$\mathbb{C}^\times\subset H$ as $t\mapsto t^{n-2}$ for odd~$n$ and
as $t\mapsto t^{\frac{n-2}2}$ for even~$n$.

We search for functions $f\in \mathbb{C}[G]$ satisfying
$$
f(P,Q)=f(u_1^{-1}Ph_1h_2,u_2^{-1}Qh_2h_3)
$$
for all $P\in G_1$, $Q\in G_2$, $u_1\in U_1$, $u_2\in U_2$, $h_i\in
H_i$, $i=1,2,3$. Suppose $P=(p_{ij})$ and $Q=(q_{ij})$. We denote
by~$P_{ij}$ the $(i,j)$-cofactor of $P$ so that $P^{-1}=(P_{ji})$.

Theorem~\ref{th6} allows us not to consider the action of~$H_3$ and
to reduce to canonical form the first and last columns of~$Q$.

Suppose $\Delta=p_{n-1,n-1}p_{n,n}-p_{n-1,n}p_{n,n-1}$,
$W=q_{2m-1,1}q_{2m,2m}-q_{2m-1,2m}q_{2m,1}$,
$D=p_{n,n-1}q_{2m,2m}-p_{n,n}q_{2m,1}$,
$\Phi_1=p_{n,n-1}P_{1,n-1}+p_{n,n}P_{1,n}$,
$\Phi_2=q_{2m,1}P_{1,n-1}+q_{2m,2m}P_{1,n}$, $\delta$ is the minor
of~$P$ corresponding to the last $n-2$ rows and the first $n-2$
columns. We have~$\Delta$, $W$, $D$, $\Phi_1$, $\Phi_2$, $\delta\in
A$. At that, $W\equiv1$ for $m=1$ and $\Phi_1\equiv-\delta\Delta$
for $n=3$. All these functions are weight functions with respect to
$B\times H$. Their weights are listed in Table~\ref{tab2}.

\begin{table}[h]
\renewcommand{\arraystretch}{1.2}
\renewcommand{\tabcolsep}{2pt}
\begin{center}
\caption{} \vskip3mm \label{tab2} \footnotesize{
\begin{tabular}{|c|c|c|c|c|c|c|}
\hline No. & $\Delta$ & $W$ & $D$ & $\Phi_1$ & $\Phi_2$ & $\delta$
\\
\hline 3 & $(\pi_{n-2},2\chi_0)$ & $(\varphi_2,0)$ ($m\ge2$) &
$(\pi_{n-1}{+}\varphi_1,\chi_0)$ & $(\pi_1{+}\pi_{n-1},0)$ &
$(\pi_1{+}\varphi_1,-\chi_0)$ & $(\pi_2,-2\chi_0)$
\\
\hline 4 & $\pi_{n-2}$ & $\varphi_2$ ($m\ge2$) &
$\pi_{n-1}{+}\varphi_1$ & $\pi_1{+}\pi_{n-1}$ & $\pi_1{+}\varphi_1$
& $\pi_2$
\\
\hline
\end{tabular}
}
\end{center}
\end{table}

Below we shall apply Theorem~\ref{th7} to the
functions~$\Delta$,~$W$ (the latter is present for $m\ge2$), $D$,
$\Phi_1$, $\Phi_2$, $\delta$ for $n\ge4$ (in~Cases~3,~4) and the
functions~$\Delta$,~$W$ (the latter is present for $m\ge2$), $D$,
$\Phi_2$, $\delta$ for $n=3$ (in~Case~3). We note that in each case
the set of weights with respect to $B\times H$ corresponding to
these functions is linearly independent.

Let $M$ be the open subset of~$G$ given by $\Delta\ne0$, $W \ne 0$,
$D\ne0$, $\Phi_1\ne0$, $\delta\ne0$ and suppose $(P,Q)\in M$. Acting
by~$H_2$ we transform $(P,Q)$ to a pair $(P',Q')$ such that the
lower right $2\times2$ block of~$P'$ and the lower $2\times2$ block
of~$\overline Q{}'$ are
$$
\begin{pmatrix}
1 & *
\\
0 & \Delta
\end{pmatrix},\qquad
\begin{pmatrix}
* & \dfrac{W\Delta}{D}
\\
-\dfrac{D}{\Delta} & 0
\end{pmatrix}
$$
respectively. Then acting by~$U_2$ (Lemma~\ref{lem4}) we
transform~$Q'$ to a matrix~$Q''$ where
$$
\overline Q{}''=\begin{pmatrix} 0 & \dfrac{\Delta}{D}
\\
\hdots & \hdots
\\
0 & \dfrac{W\Delta}{D}
\\
-\dfrac{D}{\Delta} & 0
\end{pmatrix},\quad
m\ge2,\qquad \overline Q{}''=
\begin{pmatrix}
0 & \dfrac{\Delta}{D}
\\
-\dfrac{D}{\Delta} & 0
\end{pmatrix},\quad
m=1.
$$
(The dots stand for zero entries.) We now turn to the matrix~$P'$.
Firstly, acting by~$U_1$ we make all the entries in the last two
columns equal to zero except for the two on the diagonal of the
lower $2\times2$ block. (These entries are~$1$ and~$\Delta$.) Acting
by~$H_1$ on the obtained matrix, we transform the block
corresponding to the last $n-2$ rows and first $n-2$ columns to the
form $\diag(1,\dots,1,\delta)$. After that, again acting by~$U_1$,
we transform the new matrix to the form (for $n\ge5$, $n=4$, $n=3$,
respectively)
\begin{equation}
\label{eq5}
\begin{pmatrix}
0 & \hdots & 0 & \pm\dfrac{\delta}{\Phi_1} & 0 & 0
\\[3mm]
0 & \hdots & \pm\dfrac{\Phi_1}{\Delta\delta} & \pm\dfrac{\Phi_2}{D}
& 0 & 0
\\[2mm]
1 & \hdots & 0 & 0 & 0 & 0
\\
\hdots & \hdots & \hdots & \hdots & \hdots & \hdots
\\
0 & \hdots & 1 & 0 & 1 & 0
\\
0 & \hdots & 0 & \delta & 0 & \Delta
\end{pmatrix},\quad
\begin{pmatrix}
0 & -\dfrac{\delta}{\Phi_1} & 0 & 0
\\[3mm]
\dfrac{\Phi_1}{\Delta\delta} & \dfrac{\Phi_2}{D} & 0 & 0
\\[3mm]
1 & 0 & 1 & 0
\\
0 & \delta & 0 & \Delta
\end{pmatrix},\quad
\begin{pmatrix}
\dfrac{1}{\Delta}& 0 & 0
\\[3mm]
\dfrac{\Phi_2}D & 1 & 0
\\[2mm]
\delta & 0 & \Delta
\end{pmatrix}.
\end{equation}
(The lower left $(2n-2)\times(2n-2)$ block of the first matrix is
equal to $\diag(1,\dots\allowbreak\dots,1,\delta)$ and the remaining
dots stand for zero entries.) We denote the resulting
matrix~\eqref{eq5} by~$P''$ in each of the cases $n\ge5$, $n=4$,
$n=3$.

The pair $(P'',Q'')$ is the canonical form of the pair $(P,Q)$. Thus
the section is obtained. Therefore the desired algebra~$A$ is
contained in the algebra
\begin{equation}
\label{eq6} \tilde
A=\mathbb{C}\biggl[\frac{\Delta}{D}\,,\frac{D}{\Delta}\,,
\frac{W\Delta}{D}\,,\Delta,\delta,\frac{\Phi_1}{\Delta\delta}\,,
\frac{\Phi_2}{D}\,,\frac{\delta}{\Phi_1}\biggr] \subset
\mathbb{C}\biggl[\Delta,W,D,\Phi_1,\Phi_2,\delta,
\frac{1}{\Delta}\,,\frac{1}{D}\,,\frac{1}{\Phi_1}\,,\frac{1}{\delta}\biggr]\
\
\end{equation}
for $n\ge4$ and in the algebra
\begin{equation}
\label{eq7} \tilde
A=\mathbb{C}\biggl[\frac{\Delta}{D}\,,\frac{D}{\Delta}\,,
\frac{W\Delta}{D}\,,\Delta,\delta,\frac{\Phi_2}{D}\,,\frac{1}{\Delta}\biggr]
\subset \mathbb{C}\biggl[\Delta,W,D,\Phi_2,\delta,
\frac{1}{\Delta}\,,\frac{1}{D}\biggr]
\end{equation}
for $n=3$.

We now apply Theorem~\ref{th7} to the set of functions~$\Delta$,
$D$, $\Phi_1$, $\delta$, $\Phi_2$,~$W$ (the last is present for
$m\ge2$) for $n\ge4$ and~$\Delta$, $D$, $\delta$, $\Phi_2$,~$W$ (the
last is present for $m\ge2$) for $n=3$. We have already seen that
condition~a) holds. Condition~b) follows from the
inclusions~\eqref{eq6} and~\eqref{eq7}.

Further we use assertion~1) of Theorem~\ref{th7}. First, we note
that the weights of~$\Delta$, $W$ (for $m\ge2$), $\delta$ admit no
representation of the form~\eqref{eq3}. Hence, these three functions
are irreducible in~$A$. Every representation of the weight of~$D$ in
the form~\eqref{eq3} has the form
$(\pi_{n-1},a\chi_0)+(\varphi_1,b\chi_0)$ where $a,b\in\mathbb{Z}$
and $a+b=1$. Every representation of the weight of~$\Phi_1$ in the
form~\eqref{eq3} has the form $(\pi_1,a\chi_0)+(\pi_{n-1},b\chi_0)$
where $a,b\in\mathbb{Z}$ and $a+b=0$. Every representation of the
weight of~$\Phi_2$ in the form~\eqref{eq3} has the form
$(\pi_1,a\chi_0)+(\varphi_1,b\chi_0)$ where $a,b\in\mathbb{Z}$ and
$a+b=-1$. We now distinguish two possibilities: $n\ge4$ and $n=3$.

At first, suppose $n\ge4$. Then it is easy to see that none of the
weights of the form $(\pi_1,p\chi_0)$, $(\pi_{n-1},q\chi_0)$,
$(\varphi_1,r\chi_0)$, where $p,q,r\in\mathbb{Z}$, lies in the set
\begin{align*}
Z&=\bigl\langle(\pi_{n-2},2\chi_0), (\pi_{n-1}+\varphi_1,\chi_0),
(\pi_1+\pi_{n-1},0),(\pi_2,-2\chi_0)\bigr\rangle
\\
&\qquad+S\bigl\langle(\pi_1+\varphi_1,-\chi_0),(\varphi_2,0)\bigr\rangle.
\end{align*}
Therefore the functions~$D$, $\Phi_1$, $\Phi_2$ are also irreducible
in~$A$. Thus we have checked the hypothesis of assertion~3) of
Theorem~\ref{th7}, hence the functions~$\Delta$, $D$, $\Phi_1$,
$\delta$, $\Phi_2$,~$W$ ($m\ge2$) generate the algebra~$A$ and their
weights with respect to $B\times H$ generate the
semigroup~$\widehat\Gamma(G/H)$.

Now suppose $n=3$. We shall prove that the function~$D$ is
irreducible using assertion~1) of Theorem~\ref{th7}. Assume that
there are integers~$a$,~$b$ such that $a+b=1$ and both weights
$(\pi_2,a\chi_0)$, $(\varphi_1,b\chi_0)$ lie in the set
$$
Z=\bigl\langle(\pi_1,2\chi_0),(\pi_2+\varphi_1,\chi_0)\bigr\rangle+
S\bigl\langle(\pi_2,-2\chi_0),(\pi_1+\varphi_1,-\chi_0),(\varphi_2,0)\bigr\rangle.
$$
Then it is not hard to show that
\begin{align*}
(\pi_2,a\chi_0)&=p(\pi_1,2\chi_0)+p(\pi_2+\varphi_1,\chi_0)+
(1-p)(\pi_2,-2\chi_0)-p(\pi_1+\varphi_1,-\chi_0),
\\
(\varphi_1,b\chi_0)&=-q(\pi_1,2\chi_0)+(1-q)(\pi_2+\varphi_1,\chi_0)+
(q-1)(\pi_2,-2\chi_0)
\\
&\qquad+q(\pi_1+\varphi_1,-\chi_0)
\end{align*}
for some integers~$p$,~$q$. At that, $1-p\ge0$, $-p\ge0$, $q-1\ge0$,
$q\ge0$, whence $p\le0$ and $q\ge1$. Next, we have $a=6p-2$ and
$b=-6q+3$. Since $a+b=1$, it follows that $p=q$. This contradicts
the inequalities $p\le0$ and $q\ge1$, thus the function~$D$ is
irreducible in~$A$. Now let us show that the function~$\Phi_2$
satisfies the hypothesis of assertion~2) of Theorem~\ref{th7}.
First, arguing as for the weight of~$D$ we obtain that
$(\pi_1+\varphi_1,-\chi_0)=(\pi_1,2\chi_0)+(\varphi_1,-3\chi_0)$ is
the unique representation of the weight of~$\Phi_2$ in the
form~\eqref{eq3} such that both summands lie in~$Z$. At that,
$(\pi_1,2\chi_0)$ is the weight of~$\Delta$. Second, we consider the
matrices $P=-F_3\in\oSL_3$ and $Q=E_{2m}\in\oSp_{2m}$. We have
$\Delta(P,Q)=0$, $\Phi_2(P,Q)=-1\ne0$, whence~$\Phi_2$ is not
divisible by~$\Delta$. Therefore $\Phi_2$ is irreducible. Hence by
assertion~3) of Theorem~\ref{th7} the functions~$\Delta$, $D$,
$\delta$, $\Phi_2$, $W$ ($m\ge2$) generate the algebra~$A$ and their
weights with respect to $B\times H$ generate the
semigroup~$\widehat\Gamma(G/H)$.

\textsl{Case}~5. $G=\oSp_{2n}\times \oSp_{2m},
H=H_0=\oSp_{2n-2}\times \oSp_2\times \oSp_{2m-2},
\mathfrak{X}(H)=0.$ The factor $\oSp_{2n-2}$ of~$H$ is embedded in
the factor $\oSp_{2n}$ of~$G$ as the central $(2n-2)\times(2n-2)$
block. Similarly the factor $\oSp_{2m-2}$ is embedded in the
factor~$\oSp_{2m}$ as the central $(2m-2)\times(2m-2)$ block. The
factor $\oSp_2$ of~$H$ is diagonally embedded in~$G$ as the
$2\times2$ block in the first and last rows and columns in both
factors of~$G$.

We are interested in functions $f(P,Q)\in \mathbb{C}[G]$ such that
$$
f(P,Q)=f(u_1^{-1}Ph_1h_2,u_2^{-1}Qh_2h_3)
$$
for all matrices $P\in G_1$, $Q\in G_2$, $u_1\in U_1$, $u_2\in U_2$,
$h_i\in H_i$, $i=1,2,3$. Suppose $P=(p_{ij})$, $Q=(q_{ij})$.

Theorem~\ref{th6} allows us not to consider the actions of~$H_1$
and~$H_3$ and to reduce to canonical form only the first and last
columns of~$P$ and~$Q$.

We introduce the functions
$\Delta=p_{2n-1,1}p_{2n,2n}-p_{2n-1,2n}p_{2n,1}$,
$\delta=q_{2m-1,1}q_{2m,2m}-q_{2m-1,2m}q_{2m,1}$,
$D=p_{2n,1}q_{2m,2m}-p_{2n,2m}q_{2m,1}$. We have $\Delta\equiv 1$
for $n=1$ and $\delta\equiv1$ for $m=1$. It is clear that~$\Delta$,
$\delta$,~$D$ lie in~$A$ and are weight functions with respect to
$B\times H$. Their weights are equal to $\pi_2$ ($n\ge2$),
$\varphi_2$ ($m\ge2$), $\pi_1+\varphi_1$, respectively. Below we
shall apply Theorem~\ref{th7} to these functions.

We consider the open subset $M\subset G$ determined by the
conditions $\Delta\,{\ne}\,0$, $\delta \ne 0$,  $D\ne0$. Suppose
$(P,Q)\in M$.

Acting by~$H_2$ we transform $(P,Q)$ to a pair $(P',Q')$ such that
the lower $2\times 2$ blocks of the matrices~$\overline P{}'$
and~$\overline Q{}'$ are
$$
\begin{pmatrix}
1 & *
\\
0 & \Delta
\end{pmatrix},\qquad
\begin{pmatrix}
* & \dfrac{\delta\Delta}D
\\
-\dfrac D\Delta & 0
\end{pmatrix},
$$
respectively. Then, acting by~$U_1$ and~$U_2$ (Lemma~\ref{lem4}) we
transform~$P'$,~$Q'$ to~$P''$,~$Q''$, respectively, where
$$
\overline P{}''=
\begin{pmatrix}
0 & 0
\\
0 & -1
\\
\hdots & \hdots
\\
1 & 0
\\
0 & \Delta
\end{pmatrix},\qquad
\overline Q{}''=
\begin{pmatrix}
0 & \dfrac\Delta{D}
\\
\hdots & \hdots
\\
0 & \dfrac{\delta\Delta}D
\\
-\dfrac D\Delta & 0
\end{pmatrix}
$$
for $n,m\ge2$. In these matrices the dots stand for zero entries. If
$n=1$ or $m=1$ then
$$
P''=\overline P{}''=\begin{pmatrix} 1 & 0
\\
0 & 1
\end{pmatrix},\qquad
Q''=\overline Q{}''=\begin{pmatrix} 0 & \dfrac\Delta D
\\
-\dfrac D\Delta & 0
\end{pmatrix}
$$
respectively.

The pair $(P'',Q'')$ is the canonical form of the pair $(P,Q)$. Thus
the section is obtained, therefore the desired algebra~$A$ is
contained in the algebra
$$
\tilde
A=\mathbb{C}\biggl[\Delta,\frac{\Delta}{D}\,,\frac{D}{\Delta}\,,
\frac{\delta\Delta}{D}\biggr]\subset
\mathbb{C}\biggl[\Delta,\delta,D,\frac{1}{\Delta}\,,\frac{1}{D}\biggr].
$$
By assertions~1),~3) of Theorem~\ref{th7}, the functions $\Delta$
($n \ge2$), $\delta$ ($m \ge2$), $D$ are irreducible and generate
the algebra~$A$ and their weights with respect to $B\times H$
generate the semigroup~$\widehat\Gamma(G/H)$.

\textsl{Case}~6. $G=\oSp_{2n}\times \oSp_4$,
$H=H_0=\oSp_{2n-4}\times \oSp_4$, $\mathfrak{X}(H)=0$. The first
factor of~$H$ is embedded in the first factor of~$G$ as the central
$(2n-4)\times(2n-4)$ block; the second factor of~$H$ is diagonally
embedded in~$G$, as the $4\times 4$ block in rows and columns
nos.~$1$, $2$, $2n-1$, $2n$ in the first factor.

We are interested in functions $f(P,Q)\in \mathbb{C}[G]$ such that
$$
f(P,Q)=f(u_1^{-1}Ph_1h_2,u_2^{-1}Qh_2)
$$
for all matrices $P\in G_1$, $Q\in G_2$, $u_1\in U_1$, $u_2\in U_2$,
$h_1\in H_1$, $h_2\in H_2$.

Suppose $(P,Q)\in G$ is an arbitrary pair of matrices, the set~$M$
will be chosen later. Let us reduce this pair to canonical form.
First of all, we put $h_2=Q^{-1}u_2$ and thereby transform~$Q$ to
the identity matrix~$E_4$. Now the problem is reduced to finding a
canonical form for the matrix $PQ^{-1}\in\oSp_{2n}$ with respect to
the right action of~$U_1$ and left actions of~$U_2$,~$H_1$. For
short, we put $PQ^{-1}=R$. Suppose $R=(r_{ij})$. By
Theorem~\ref{th6} we do not consider anymore the action of~$H_1$ and
restrict ourselves to the problem of reducing only the first two and
the last two columns of~$R$ to canonical form.

We denote by~$\Delta_i$, $i=1,2,3,4$, the minor of~$\overline R$
corresponding to the last~$i$ rows and first~$i$ columns. Let $\Phi$
be the minor of order~3 of~$\overline R$ corresponding to the last
three rows and columns~$1$,~$2$,~$4$. Next, we put
$D=r_{2n,1}r_{2n-1,2n}-
r_{2n-1,1}r_{2n,2n}+r_{2n,2}r_{2n-1,2n-1}-r_{2n-1,2}r_{2n,2n-1}$,
$F=\Delta_1\Phi+r_{2n,2}\Delta_3$. Below we shall find out that
$\Delta_4=-D$ for $n=3$. The functions~$\Delta_1$, $\Delta_2$,
$\Delta_3$, $\Delta_4$,~$D$, and~$F$ lie in~$A$ and are weight
functions with respect to $B\times H$. Their weights are equal to
$\pi_1+\varphi_1$, $\pi_2+\varphi_2$, $\pi_3+\varphi_1$, $\pi_4$
($n\ge4$), $\pi_2$, and $\pi_1+\pi_3+\varphi_2$, respectively. Below
we shall apply Theorem~\ref{th7} to the functions~$\Delta_1$,
$\Delta_2$, $\Delta_3$, $\Delta_4$, $D$, $F$ for $n\ge4$ and
functions~$\Delta_1$, $\Delta_2$, $\Delta_3$, $D$, $F$ for $n=3$.

The following argument will first be performed for $n\ge 4$. We
shall reduce the matrix~$R$ to canonical form on the open subset
$M\subset G$ where $\Delta_i\ne0$ for $i=1,2,3,4$.

Using the left action of~$U_1$ by matrices of type~\eqref{eq1} we
transform~$R$ to a matrix~$R'$ such that all the non-zero elements
of the lower half of the matrix~$\overline R{}'$ are concentrated in
its lower $4\times 4$ block which has the form
$$
\begin{pmatrix}
0 & 0 & 0 & -\dfrac{\Delta_4}{\Delta_3}
\\
0 & 0 & \dfrac{\Delta_3}{\Delta_2} & r_6
\\
0 & -\dfrac{\Delta_2}{\Delta_1} & r_4 & r_5
\\
\Delta_1 & r_1 & r_2 & r_3
\end{pmatrix},
$$
where $r_1=r_{2n,2}$, $r_2=r_{2n,2n-1}$, $r_3=r_{2n,2n}$,
$r_4=\cfrac{r_{2n,1}r_{2n-1,2n-1}-r_{2n-1,1}r_{2n,2n-1}}{\Delta_1}$\,,
$r_5=\cfrac{r_{2n,1}r_{2n-1,2n}-r_{2n-1,1}r_{2n,2n}}{\Delta_1}$\,,
$r_6=\cfrac{\Phi}{\Delta_2}$\,.

Multiplying~$R'$ on the right by an appropriate matrix in~$U_2$ of
type~\eqref{eq1} and then by an appropriate matrix in~$U_2$ of
type~\eqref{eq2} we successively obtain two matrices~$R'_1$
and~$R'_2$ such that the lower $4\times 4$ blocks of~$\overline
R{}'_1$ and~$\overline R{}'_2$ are
$$
\begin{pmatrix}
0 & 0 & 0 & -\dfrac{\Delta_4}{\Delta_3}
\\[3mm]
0 & 0 & \dfrac{\Delta_3}{\Delta_2} &
r_6+\dfrac{r_1\Delta_3}{\Delta_1\Delta_2}
\\[3mm]
0 & -\dfrac{\Delta_2}{\Delta_1} & r_4 & r_5+\dfrac{r_1r_4}{\Delta_1}
\\[3mm]
\Delta_1 & 0 & r_2 & r_3+\dfrac{r_1r_2}{\Delta_1}
\end{pmatrix},\qquad
\begin{pmatrix}
0 & 0 & 0 & -\dfrac{\Delta_4}{\Delta_3}
\\[3mm]
0 & 0 & \dfrac{\Delta_3}{\Delta_2} & \dfrac{F}{\Delta_1\Delta_2}
\\[3mm]
0 & -\dfrac{\Delta_2}{\Delta_1} & 0 & \dfrac{D}{\Delta_1}
\\[3mm]
\Delta_1 & 0 & 0 & 0
\end{pmatrix},
$$
respectively. Further, acting on the left by matrices in~$U_1$ of
type~\eqref{eq2} we transform~$R'_2$ to a matrix~$R''$ where
$$
\overline R{}''=\begin{pmatrix} 0 & 0 & 0 & -\dfrac{1}{\Delta_1}
\\
0 & 0 & \dfrac{\Delta_1}{\Delta_2} & 0
\\
0 & 0 & 0 & \dfrac{D}{\Delta_3}
\\
\hdots & \hdots & \hdots & \hdots
\\
0 & 0 & 0 & -\dfrac{\Delta_4}{\Delta_3}
\\[3mm]
0 & 0 & \dfrac{\Delta_3}{\Delta_2} & \dfrac{F}{\Delta_1\Delta_2}
\\[3mm]
0 & -\dfrac{\Delta_2}{\Delta_1} & 0 & \dfrac{D}{\Delta_1}
\\
\Delta_1 & 0 & 0 & 0
\end{pmatrix}.
$$
(The dots stand for zero entries.) The pair $(R'',E_4)$ is the
canonical form of the original pair $(P,Q)$. Thus the section is
obtained. Therefore the desired algebra~$A$ is contained in the
algebra
\begin{align*}
\tilde A&=\mathbb{C}\biggl[\Delta_1,\frac{\Delta_2}{\Delta_1}\,,
\frac{\Delta_3}{\Delta_2}\,,\frac{\Delta_4}{\Delta_3}\,,\frac{D}{\Delta_1}\,,
\frac{F}{\Delta_1\Delta_2}\,,\frac{D}{\Delta_3}\,,\frac{\Delta_1}{\Delta_2}\,,
\frac{1}{\Delta_1}\biggr]
\\
&\subset\mathbb{C}\biggl[\Delta_1,\Delta_2,\Delta_3,\Delta_4,D,F,
\frac{1}{\Delta_1}\,,\frac{1}{\Delta_2}\,,\frac{1}{\Delta_3}\biggr].
\end{align*}
By assertions~1),~3) of\, Theorem~\ref{th7} all the
functions~$\Delta_1$, $\Delta_2$, $\Delta_3$, $\Delta_4$, $D$,~$F$
are irreducible and generate~$A$ and their weights generate the
semigroup~$\widehat\Gamma(G/H)$.

For $n=3$ a matrix~$R$ in the open subset $M=\{\Delta_1\ne0, \
\Delta_2\ne\nobreak0, \allowbreak \Delta_3\ne\nobreak0\}\subset G$
can similarly be transformed to a new matrix~$R'$ where
$$
\overline R{}'=\begin{pmatrix} 0 & 0 & 0 & -\dfrac{1}{\Delta_1}
\\
0 & 0 & \dfrac{\Delta_1}{\Delta_2} & 0
\\
0 & 0 & 0 & \dfrac{D}{\Delta_3}
\\[3mm]
0 & 0 & \dfrac{\Delta_3}{\Delta_2} & \dfrac{F}{\Delta_1\Delta_2}
\\[3mm]
0 & -\dfrac{\Delta_2}{\Delta_1} & 0 & \dfrac{D}{\Delta_1}
\\
\Delta_1 & 0 & 0 & 0
\end{pmatrix}.
$$
From this it follows that $\Delta_4=-D$. The pair $(R',E_4)$ is the
canonical form of the original pair $(P,Q)$. Thus the section is
obtained. Therefore
\begin{align*}
A\subset \tilde
A&=\mathbb{C}\biggl[\Delta_1,\frac{\Delta_2}{\Delta_1}\,,
\frac{\Delta_3}{\Delta_2}\,,\frac{D}{\Delta_3}\,,\frac{F}{\Delta_1\Delta_2}\,,
\frac{D}{\Delta_1}\,,\frac{\Delta_1}{\Delta_2}\,,\frac{1}{\Delta_1}\biggr]
\\
&\subset\mathbb{C}\biggl[\Delta_1,\Delta_2,\Delta_3,\Delta_4,D,F,
\frac{1}{\Delta_1}\,,\frac{1}{\Delta_2}\,,\frac{1}{\Delta_3}\biggr].
\end{align*}
By assertions~1),~3) of Theorem~\ref{th7} all the
functions~$\Delta_1$, $\Delta_2$, $\Delta_3$, $D$,~$F$ are
irreducible and generate the algebra~$A$ and their weights generate
the semigroup~$\widehat\Gamma(G/H)$.

\textsl{Case}~7. $G=\oSp_{2n}\times \oSp_{2m}\times \oSp_{2l}$,
$H=H_0=\oSp_{2n-2}\times \oSp_{2m-2}\times \allowbreak
\oSp_{2l-2}\times \oSp_2$, $\mathfrak{X}(H)=0$. The first three
factors of~$H$ are embedded in the corresponding factors of~$G$ as
the central blocks of the appropriate size. The factor $\oSp_2$
of~$H$ is diagonally embedded in~$G$ as the $2\times 2$ block in the
first and last rows and columns of each factor.

We search for functions $f(P,Q,R)\in \mathbb{C}[G]$ such that
$$
f(P,Q,R)=f(u_1^{-1}Ph_1h_4,u_2^{-1}Qh_2h_4,u_3^{-1}Rh_3h_4)
$$
for all $P\in G_1$, $Q\in G_2$, $R\in G_3$, $u_i\in U_i$, $i=1,2,3$,
$h_j\in H_j$, $j=1,2,3,4$. Based on Theorem~\ref{th6}, we may
discard the actions of the first three factors of~$H$ and reduce
only the first and last columns of~$P$, $Q$, $R$ to canonical form.
Suppose $P=(p_{ij})$, $Q=(q_{ij})$, $R=(r_{ij})$.

Let us introduce the following functions:
$\Delta_1=p_{2n-1,1}p_{2n,2n}-p_{2n-1,2n}p_{2n,1}$,
$\Delta_2=q_{2m-1,1}q_{2m,2m}-q_{2m-1,2m}q_{2m,1}$,
$\Delta_3=r_{2l-1,1}r_{2l,2l}-r_{2l-1,2l}r_{2l,1}$,
$D_1=p_{2n,1}q_{2m,2m}-p_{2n,2n}q_{2m,1}$,
$D_2=q_{2m,1}r_{2l,2l}-q_{2m,2m}r_{2l,1}$,
$D_3=p_{2n,1}r_{2l,2l}-p_{2n,2n}r_{2l,1}$. If $n=1$ (resp. $m=1$,
$l=1$) then $\Delta_1\equiv1$ (resp. $\Delta_2\equiv1$,
$\Delta_3\equiv1$). All the functions~$\Delta_1$, $\Delta_2$,
$\Delta_3$, $D_1$, $D_2$,~$D_3$ lie in~$A$ and are weight functions
with respect to $B\times H$. Their weights are~$\pi_2$ ($n\ge2$),
$\varphi_2$ ($m\ge2$), $\psi_2$~($l\ge2$), $\pi_1+\varphi_1$,
$\varphi_1+\psi_1$, $\pi_1+\psi_1$, respectively. Below we shall
apply Theorem~\ref{th7} to these functions.

We consider the open subset $M\,{\subset}\, G$ determined by the
conditions $\Delta_1{\ne}\,0$, $\Delta_2{\ne}\,0$,
$\Delta_3{\ne}\,0$, $D_1\ne0$, $D_3\ne0$. Acting on the triple
$(P,Q,R)\in M$ by an appropriate matrix in~$H_4$ we obtain a new
triple $(P',Q',R')$ such that the lower $2\times 2$ blocks
of~$\overline P{}'$, $\overline Q{}'$, $\overline R{}'$ have the
form
\begin{gather*}
\begin{pmatrix}
1 & *
\\
0 & \Delta_1
\end{pmatrix},\qquad
\begin{pmatrix}
* & \dfrac{\Delta_2\Delta_1}{D_1}
\\
-\dfrac{D_1}{\Delta_1} & 0
\end{pmatrix},
\\
\begin{pmatrix}
\dfrac{r_{2l-1,1}p_{2n,2n}-r_{2l-1,2l}p_{2n,1}}{\Delta_1} &
(r_{2l-1,1}q_{2m,2m}-r_{2l-1,2l}q_{2m,1})\dfrac{\Delta_1}{D_1}
\\[3mm]
-\dfrac{D_3}{\Delta_1} & -\dfrac{D_2\Delta_1}{D_1}
\end{pmatrix},
\end{gather*}
respectively. Now acting by~$U$ (Lemma~\ref{lem4}) we transform the
triple $(P',Q',R')$ to a triple $(P'',Q'',R'')$ where for
$n,m,l\ge2$
\begin{gather*}
\overline P{}''=\begin{pmatrix} 0 & 0
\\
0 & -1
\\
\hdots & \hdots
\\
1 & 0
\\
0 & \Delta_1
\end{pmatrix},\qquad
\overline Q{}''=\begin{pmatrix} 0 & \dfrac{\Delta_1}{D_1}
\\
\hdots & \hdots
\\
0 & \dfrac{\Delta_2\Delta_1}{D_1}
\\
-\dfrac{D_1}{\Delta_1} & 0
\end{pmatrix},
\qquad \overline R{}''=\begin{pmatrix} 0 & \dfrac{\Delta_1}{D_3}
\\[3mm]
0 & 0
\\
\hdots & \hdots
\\
0 & \dfrac{\Delta_1\Delta_3}{D_3}
\\[3mm]
-\dfrac{D_3}{\Delta_1} & -\dfrac{D_2\Delta_1}{D_1}
\end{pmatrix}
\end{gather*}
(The dots stand for zero entries.) In the cases $n=1$, $m=1$, $l=1$
we have
$$
\overline P{}''=\begin{pmatrix} 1 & 0
\\
0 & 1
\end{pmatrix},\qquad
\overline Q{}''=\begin{pmatrix} 0 & \dfrac{\Delta_1}{D_1}
\\
-\dfrac{D_1}{\Delta_1} & 0
\end{pmatrix},\qquad
\overline R{}''=\begin{pmatrix} 0 & \dfrac{\Delta_1}{D_3}
\\[3mm]
-\dfrac{D_3}{\Delta_1} & -\dfrac{D_2\Delta_1}{D_1}
\end{pmatrix},
$$
respectively. The triple $(P'',Q'',R'')$ is the canonical form of
the original triple $(P,\allowbreak Q,R)$. Thus the section is
obtained. Therefore the algebra~$A$ is contained in the algebra
\begin{align*}
\tilde A&=\mathbb{C}\biggl[\Delta_1,\frac{\Delta_1}{D_1}\,,
\frac{\Delta_2\Delta_1}{D_1}\,,\frac{D_1}{\Delta_1}\,,\frac{\Delta_1}{D_3}\,,
\frac{\Delta_1\Delta_3}{D_3}\,,\frac{D_3}{\Delta_1}\,,
\frac{D_2\Delta_1}{D_1}\biggr]
\\
&\subset \mathbb{C}\biggl[\Delta_1,\Delta_2,\Delta_3,D_1, D_2,
D_3,\frac{1}{\Delta_1}\,,\frac{1}{D_1}\,,\frac{1}{D_3}\biggr].
\end{align*}

By assertions~1),~3) of Theorem~\ref{th7} the functions~$\Delta_1$
($n\ge2$), $\Delta_2$ ($m\ge\nobreak2$), $\Delta_3$~($l\ge2$),
$D_1$, $D_2$, $D_3$ are irreducible and generate~$A$ and their
weights generate the semigroup~$\widehat\Gamma(G/H)$.

\textsl{Case}~8. $G=\oSp_{2n}\times \oSp_4\times \oSp_{2m}$,
$H=H_0=\oSp_{2n-2}\times \oSp_2\times \oSp_2\times\allowbreak
\oSp_{2m-2}$, $\mathfrak{X}(H)=0$. The subgroup~$H$ is embedded
in~$G$ as follows. The factor~$H_1$ is embedded in~$G_1$ as the
central $(2n-2)\times(2n-2)$ block. The factor~$H_4$ is similarly
embedded in~$G_3$. The factor~$H_2$ is diagonally embedded in~$G_1$
and~$G_2$ as the $2\times 2$ block in the first and last rows and
columns. The factor~$H_3$ is diagonally embedded in~$G_2$ and~$G_3$,
as the central $2\times 2$ block in~$G_2$ and as the $2\times 2$
block in the first and last rows and columns in~$G_3$.

We are interested in those functions $f(P,Q,R)\in \mathbb{C}[G]$
that satisfy
$f(P,Q,R)=f(u_1^{-1}Ph_1h_2,u_2^{-1}Qh_2h_3,u_3^{-1}Rh_3h_4)$ for
all $P\in G_1$, $Q\in G_2$, $R\in G_3$, $u_i\in U_i$, $i=1,2,3$,
$h_j\in H_j$, $j=1,2,3,4$. Suppose $P=(p_{ij})$, $Q=(q_{ij})$,
$R=(r_{ij})$. Using Theorem~\ref{th6} we may discard the actions
of~$H_1$ and~$H_4$ and reduce only the first and last columns
of~$P$,~$R$ to canonical form.

We introduce the following functions:
$\Delta_1=q_{31}q_{44}-q_{34}q_{41}$,
$\Delta_2=q_{32}q_{43}-q_{33}q_{42}$,
$\delta_1=p_{2n,1}q_{44}-p_{2n,2n}q_{41}$,
$\delta_2=r_{2m,1}q_{43}-r_{2m,2m}q_{42}$,
$D_1=p_{2n-1,1}p_{2n,2n}-p_{2n-1,2n}p_{2n,1}$,
$D_2=r_{2m-1,1}r_{2m,2m}-r_{2m-1,2m}r_{2m,1}$,
$\Delta=\delta_2(p_{2n,1}q_{34}-p_{2n,2n}q_{31})-
\delta_1(r_{2m,1}q_{33}-r_{2m,2m}q_{32})$. We have $D_1\equiv1$ for
$n=1$, $D_2\equiv1$ for $m=1$. Since the last two columns of~$Q$ are
skew-orthogonal, it follows that $\Delta_2=-\Delta_1$. The
functions~$\Delta_1$, $\delta_1$, $\delta_2$, $D_1$, $D_2$,~$\Delta$
lie in~$A$ and are weight functions with respect to $B\times H$,
their weights are~$\varphi_2$, $\pi_1+\varphi_1$,
$\varphi_1+\psi_1$, $\pi_2$ ($n\ge2$), $\psi_2$ ($m\ge2$),
$\pi_1+\varphi_2+\psi_1$, respectively. Below we shall apply
Theorem~\ref{th7} to these functions.

Suppose $M\subset G$ is the open subset determined by the conditions
$\Delta_1\ne0$, $\delta_1\ne0$, $\delta_2\ne0$, $D_1 \ne 0$, $D_2
\ne 0$. Suppose $(P,Q,R)\in M$. Acting on this triple by appropriate
matrices in~$H_2$ and~$H_3$ we obtain a new triple of matrices
$(P',Q',R')$ where the matrix~$Q'$ and the lower $2\times 2$ blocks
of the matrices~$\overline P{}'$,~$\overline R{}'$ have the form
\begin{gather*}
\begin{pmatrix}
* & * & * & *
\\
* & * & * & *
\\
1 & 1 & (r_{2m,1}q_{33}-r_{2m,2m}q_{32})\dfrac{\Delta_2}{\delta_2} &
(p_{2n,1}q_{34}-p_{2n,2n}q_{31})\dfrac{\Delta_1}{\delta_1}
\\
0 & 0 & \Delta_2 & \Delta_1
\end{pmatrix},
\\[2mm]
\begin{pmatrix}
* & -\dfrac{D_1\Delta_1}{\delta_1}
\\
\dfrac{\delta_1}{\Delta_1} & 0
\end{pmatrix},\qquad
\begin{pmatrix}
* & -\dfrac{D_2\Delta_2}{\delta_2}
\\
\dfrac{\delta_2}{\Delta_2} & 0
\end{pmatrix},
\end{gather*}
respectively. Next, acting by~$U$ we transform the triple
($P',Q',R'$) to a triple ($P'',Q'',R''$) where for $n\ge2$ and
$m\ge2$
\begin{gather*}
\overline P{}''=\begin{pmatrix} 0 & -\dfrac{\Delta_1}{\delta_1}
\\
0 & 0
\\
\hdots & \hdots
\\
0 & -\dfrac{D_1\Delta_1}{\delta_1}
\\
\dfrac{\delta_1}{\Delta_1} & 0
\end{pmatrix},\quad
Q''=\begin{pmatrix} \dfrac{1}{\Delta_1} & 0 &
\dfrac{\Delta}{\delta_1\delta_2} & 0
\\
0 & 0 & -1 & 0
\\
1 & 1 & 0 & \dfrac{\Delta_1\Delta}{\delta_1\delta_2}
\\
0 & 0 & \Delta_2 & \Delta_1
\end{pmatrix},\quad
\overline R{}''=\begin{pmatrix} 0 & -\dfrac{\Delta_2}{\delta_2}
\\
0 & 0
\\
\hdots & \hdots
\\
0 & -\dfrac{D_2\Delta_2}{\delta_2}
\\
\dfrac{\delta_2}{\Delta_2} & 0
\end{pmatrix}.
\end{gather*}
(For the matrices~$P'$,~$R'$ this is possible by Lemma~\ref{lem4}.)
In these matrices the dots stand for zero entries. If $n=1$ or
$m=1$, then
$$
\overline P{}''=P''=\begin{pmatrix} 0 & -\dfrac{\Delta_1}{\delta_1}
\\
\dfrac{\delta_1}{\Delta_1} & 0
\end{pmatrix},\qquad
\overline R{}''=R''=\begin{pmatrix} 0 & -\dfrac{\Delta_2}{\delta_2}
\\
\dfrac{\delta_2}{\Delta_2} & 0
\end{pmatrix},
$$
respectively. The triple $(P'', Q'', R'')$ is the canonical form of
the triple $(P,Q,R)$. Thus the section is obtained. Therefore there
is an inclusion
\begin{align*}
A\subset \tilde A&=\mathbb{C}\biggl[\frac{\Delta_1}{\delta_1}\,,
\frac{D_1\Delta_1}{\delta_1}\,,\frac{\delta_1}{\Delta_1}\,,
\frac{\Delta_1}{\delta_2}\,,\frac{D_2\Delta_1}{\delta_2}\,,
\frac{\delta_2}{\Delta_1}\,,\frac{1}{\Delta_1}\,,
\frac{\Delta}{\delta_1\delta_2}\,,\frac{\Delta_1\Delta}{\delta_1\delta_2}\,,
\Delta_1\biggr]
\\
&\subset\mathbb{C}\biggl[\Delta_1,\delta_1,\delta_2,D_1,D_2,\Delta,
\frac{1}{\Delta_1}\,,\frac{1}{\delta_1}\,,\frac{1}{\delta_2}\biggr].
\end{align*}

By assertion~1) of\, Theorem~\ref{th7} the functions $\Delta_1$,
$\delta_1$, $\delta_2$, $D_1$ ($n\ge2$), $D_2$~($m\ge2$) are
irreducible. Let us prove that~$\Delta$ is also irreducible using
assertion~2) of Theorem~\ref{th7}. It is not hard to prove that
$\pi_1+\varphi_2+\psi_1=(\pi_1+\psi_1)+\varphi_2$ is the unique
representation of the weight of~$\Delta$ in the form~\eqref{eq3}
such that both summands lie in~$Z$. At that, $\varphi_2$ is the
weight of~$\Delta_1$. Consider the matrices $P=E_{2n}$,
$Q=\Omega_4$, $R=E_{2m}$. We have $(P,Q,R)\in G$,
$\Delta_1(P,Q,R)=0$, $\Delta(P,Q,R)=-1\ne0$, whence~$\Delta$ is not
divisible by~$\Delta_1$. Thus~$\Delta$ is irreducible. By
assertion~3) of Theorem~\ref{th7}, the functions~$\Delta_1$,
$\delta_1$, $\delta_2$, $D_1$ ($n\ge2$), $D_2$ ($m\ge2$), $\Delta$
generate the algebra~$A$ and their weights generate the
semigroup~$\widehat\Gamma(G/H)$.

\section*{Acknowledgements}

The author expresses his deep gratitude to E.\,B.~Vinberg for
stating the problem and numerous fruitful conversations, and also to
D.\,A.~Timashev for valuable suggestions.

\end{document}